\documentclass[leqno,11pt]{article}

\newcommand{\Dic}{\mathrm{Dic}}

\usepackage[utf8]{inputenc}
\usepackage[T1]{fontenc}
\usepackage{microtype}

\usepackage[a4paper]{geometry}

\ifdefined\screenview
  \edef\mtht{\the\textheight}
  \edef\mtwd{\the\textwidth}
\fi

\usepackage[dvipsnames]{xcolor}

\usepackage{amsmath}
\usepackage{amsthm}
\usepackage{tikz-cd}
\usetikzlibrary{arrows} 
\tikzset{
  commutative diagrams/.cd, 
  arrow style=tikz, 
  diagrams={>=stealth}
}
\usetikzlibrary{matrix,decorations.pathreplacing,calc}

\usepackage{textcomp}
\usepackage[sb]{libertine}
\usepackage[varqu,varl]{zi4}%
\usepackage[libertine,bigdelims,vvarbb]{newtxmath}
\usepackage[supstfm=libertinesups,supscaled=1.2,raised=-.13em]{superiors}
\useosf

\usepackage[scr=boondox,cal=euler]{mathalfa}

\usepackage{marvosym}

\usepackage{slashed}
\usepackage{esint} 
\usepackage[english]{babel}

\usepackage{imakeidx}
\indexsetup{noclearpage}
\makeindex[intoc]

\usepackage{csquotes}
\usepackage[
  backend=biber,
  hyperref=true,
  backref=true,
  isbn=false,
  doi=true,
  natbib=true,
  eprint=true,
  useprefix=true,
  maxcitenames=99,
  maxbibnames=99,  
  maxalphanames=99, 
  minalphanames=99,
  safeinputenc,
  style=alphabetic,
  citestyle=alphabetic,
  block=space,
  datamodel=preamble/ext-eprint,
  sorting=nyt
]{biblatex}
\usepackage[
  bookmarksnumbered = true,
  hypertexnames = false,
  colorlinks    = true,
  citecolor     = gray,
  linkcolor     = gray,
  urlcolor      = gray,
  breaklinks
]{hyperref}

\DeclareFieldFormat{url}{%
  \href{#1}{\ComputerMouse}
}
\DeclareFieldFormat{doi}{%
  \mkbibacro{DOI}\addcolon\space\href{https://doi.org/#1}{#1}
}
\makeatletter
\DeclareFieldFormat{arxiv}{%
  arXiv\addcolon\space\href{http://arxiv.org/\abx@arxivpath/#1}{#1}
}
\makeatother
\DeclareFieldFormat{mr}{%
  MR\addcolon\space\href{http://www.ams.org/mathscinet-getitem?mr=MR#1}{#1}
}
\DeclareFieldFormat{zbl}{%
  Zbl\addcolon\space\href{http://zbmath.org/?q=an:#1}{#1}
}
\renewbibmacro*{eprint}{%
  \printfield{arxiv}%
  \newunit\newblock
  \printfield{mr}%
  \newunit\newblock
  \printfield{zbl}%
  \newunit\newblock
  \iffieldundef{eprinttype}
  {\printfield{eprint}}
  {\printfield[eprint:\strfield{eprinttype}]{eprint}}
}

\AtEveryBibitem{%
  \clearlist{address}%
}
\DeclareFieldFormat[article,inproceedings,inbook,incollection,thesis]{title}{\textit{#1}}
\renewbibmacro{in:}{}
\newcommand{\printreferences}{\raggedright\printbibliography[heading=bibintoc]}

\usepackage[inline,shortlabels]{enumitem}

\usepackage{subcaption}

\usepackage[yyyymmdd]{datetime}

\usepackage{etoolbox}
\ifundef{\abstract}{}{\patchcmd{\abstract}%
    {\quotation}{\quotation\noindent\ignorespaces}{}{}}

\usepackage[super]{nth}

\usepackage{thmtools}

\numberwithin{equation}{section}

\renewcommand{\qedsymbol}{$\blacksquare$}

\newcommand{\CorollaryQED}{\qedsymbol}
\newcommand{\ConjectureQED}{$\square$}
\newcommand{\SituationQED}{$\times$}
\newcommand{\DefinitionQED}{$\bullet$}
\newcommand{\NotationQED}{$\circ$}
\newcommand{\ExampleQED}{$\spadesuit$}
\newcommand{\RemarkQED}{$\clubsuit$}

\declaretheorem[numberlike=equation,]{theorem}
\declaretheorem[numbered=no,name=Theorem]{theorem*}
\declaretheorem[numberlike=equation,name=Lemma]{lemma}
\declaretheorem[numberlike=equation,name=Proposition]{prop}

\declaretheorem[numberlike=equation,name=Situation,style=definition,qed=\SituationQED]{situation}
\declaretheorem[numberlike=equation,name=Definition,style=definition,qed=\DefinitionQED]{definition}
\declaretheorem[numbered=no,name=Definition,style=definition,qed=\DefinitionQED]{definition*}

\declaretheorem[numberlike=equation,style=definition,qed=\ExampleQED]{example}

\declaretheorem[numberlike=equation,style=remark,qed=\RemarkQED]{remark}
\declaretheorem[numbered=no,style=remark,name=Remark,qed=\RemarkQED]{remark*}

\def\makeautorefname#1#2{\AtBeginDocument{\expandafter\def\csname#1autorefname\endcsname{#2}}}
\makeautorefname{table}{Table}        
\makeautorefname{chapter}{Chapter}
\makeautorefname{section}{Section}
\makeautorefname{subsection}{Section}
\makeautorefname{subsubsection}{Section}
\makeautorefname{footnote}{Footnote}
\AtBeginDocument{\def\itemautorefname~#1\null{(#1)\null}}
\AtBeginDocument{\def\equationautorefname~#1\null{(#1)\null}}

\numberwithin{substep}{step}
\makeautorefname{step}{Step}
\makeautorefname{substep}{Step}

\makeautorefname{case}{Case}
\makeautorefname{substep}{Step}

\setlist[description]{leftmargin=!,labelindent=1em}
\setlist[enumerate]{label={\rm (\arabic*)},ref=\arabic*}
\setlist[enumerate,2]{label={\rm (\alph*)},ref=\theenumi.\alph*}
\setlist[enumerate,3]{label={\rm (\roman*)},ref=\theenumii.\roman*}

\let\C\undefined
\let\U\undefined

\usepackage{bm}
\usepackage{mathtools} 
\usepackage{stmaryrd} 

\DeclareFontFamily{U}{mathx}{\hyphenchar\font45}
\DeclareFontShape{U}{mathx}{m}{n}{
      <5> <6> <7> <8> <9> <10>
      <10.95> <12> <14.4> <17.28> <20.74> <24.88>
      mathx10
      }{}
\DeclareSymbolFont{mathx}{U}{mathx}{m}{n}
\DeclareFontSubstitution{U}{mathx}{m}{n}
\DeclareMathAccent{\widecheck}{0}{mathx}{"71}
\DeclareMathAccent{\wideparen}{0}{mathx}{"75}

\DeclareMathOperator{\Diff}{Diff}
\DeclareMathOperator{\End}{End}

\DeclareMathOperator{\GL}{GL}

\DeclareMathOperator{\HF}{\HF}

\DeclareMathOperator{\Hom}{Hom}

\DeclareMathOperator{\Isom}{Isom}

\DeclareMathOperator{\Map}{Map}

\DeclareMathOperator{\Sym}{Sym}

\DeclareMathOperator{\coker}{coker}

\DeclareMathOperator{\im}{im}

\DeclareMathOperator{\inj}{inj}

\DeclareMathOperator{\tra}{tra}

\DeclarePairedDelimiter\paren{\lparen}{\rparen}
\DeclarePairedDelimiter\sqparen{[}{]}
\DeclarePairedDelimiter{\Abs}{\|}{\|}

\DeclarePairedDelimiter{\Inner}{\langle}{\rangle}

\DeclarePairedDelimiter{\abs}{\lvert}{\rvert}
\DeclarePairedDelimiter{\bracket}{\langle}{\rangle}

\DeclarePairedDelimiter{\set}{\lbrace}{\rbrace}
\def\({\left(}
\def\){\right)}
\def\<{\left\langle}
\def\>{\right\rangle}

\newcommand{\C}{{\mathbf{C}}}
\newcommand{\Gtwo}{G_2}

\newcommand{\N}{{\mathbf{N}}}

\newcommand{\RP}{{\R P}} 

\newcommand{\R}{\mathbf{R}}

\newcommand{\SO}{\mathrm{SO}}

\newcommand{\Span}[1]{\bracket{#1}}

\newcommand{\Sp}{\mathrm{Sp}}

\newcommand{\U}{\mathrm{U}}

\newcommand{\YM}{\mathrm{YM}}
\newcommand{\Z}{\mathbf{Z}}

\newcommand{\co}{\mskip0.5mu\colon\thinspace}

\newcommand{\dR}{\mathrm{dR}}
\newcommand{\defined}[2][\key]{\def\key{#2}\textbf{#2}\index{#1}}
\newcommand{\delbar}{\bar{\del}}

\newcommand{\del}{\partial}

\newcommand{\gl}{\mathfrak{gl}}

\newcommand{\hkred}{{/\!\! /\!\! /}}

\newcommand{\id}{\mathrm{id}}
\newcommand{\incl}{\hookrightarrow}

\newcommand{\imm}{\looparrowright}

\newcommand{\iso}{\cong}

\newcommand{\one}{\mathbf{1}}

\newcommand{\pr}{\mathrm{pr}}
\newcommand{\qandq}{\quad\text{and}\quad}

\newcommand{\qforq}{\quad\text{for}\quad}
\newcommand{\qand}{\quad\text{and}}

\newcommand{\qforeveryq}{\quad\text{for every}\quad}

\newcommand{\qwithq}{\quad\text{with}\quad}

\newcommand{\vol}{\mathrm{vol}}

\renewcommand{\H}{\mathbf{H}}
\renewcommand{\Im}{\operatorname{Im}}

\renewcommand{\P}{\mathbf{P}}

\renewcommand{\emptyset}{\varnothing}
\renewcommand{\epsilon}{\varepsilon}
\renewcommand{\setminus}{{\backslash}}

\renewcommand{\leq}{\leqslant}

\makeatletter
\renewcommand*\env@matrix[1][*\c@MaxMatrixCols c]{%
  \hskip -\arraycolsep
  \let\@ifnextchar\new@ifnextchar
  \array{#1}}

\renewcommand\xleftrightarrow[2][]{%
  \ext@arrow 9999{\longleftrightarrowfill@}{#1}{#2}}
\newcommand\longleftrightarrowfill@{%
  \arrowfill@\leftarrow\relbar\rightarrow}
\makeatother

\newcommand{\rd}{{\rm d}}

\newcommand{\rH}{{\rm H}}

\newcommand{\bE}{{\mathbf{E}}}

\newcommand{\bI}{{\mathbf{I}}}

\newcommand{\bP}{{\mathbf{P}}}
\newcommand{\bQ}{{\mathbf{Q}}}

\newcommand{\bT}{{\mathbf{T}}}
\newcommand{\bU}{{\mathbf{U}}}
\newcommand{\bV}{{\mathbf{V}}}

\newcommand{\sL}{\mathscr{L}}

\newcommand{\sN}{\mathscr{N}}

\newcommand{\sS}{\mathscr{S}}
\newcommand{\sT}{\mathscr{T}}

\newcommand{\sV}{\mathscr{V}}

\newcommand{\fg}{{\mathfrak g}}

\newcommand{\fu}{{\mathfrak u}}

\newcommand{\fz}{{\mathfrak z}}

\newcommand{\fP}{{\mathfrak P}}

\newcommand{\fR}{{\mathfrak R}}

\newcommand{\bomega}{{\bm{\omega{}}}}

\newcommand{\bzeta}{{\bm\zeta}}

\author{
  Shubham Dwivedi
  \and
  Daniel Platt 
  \and
  \and
  Thomas Walpuski 
}
\title{
  Associative submanifolds in Joyce's generalised Kummer constructions
}
\date{2022-02-01}

\begin{document}

\maketitle

\begin{abstract}
  This article constructs examples of associative submanifolds in $\Gtwo$--manifolds obtained by resolving $\Gtwo$--orbifolds using Joyce's generalised Kummer construction.
  As the $\Gtwo$--manifolds approach the $\Gtwo$--orbifolds,
  the volume of the associative submanifolds tends to zero.
  This partially verifies a prediction due to \citeauthor{Halverson2015}.
\end{abstract}

\section{Introduction}
\label{Sec_Introduction}

The Teichmüller space
\begin{equation*}
  \sT(Y)
  \coloneq
  \set{
    \phi \in \Omega^3(Y) : \phi ~\textnormal{is a torsion-free $\Gtwo$--structure}
  }/\Diff_0(Y)
\end{equation*}
of torsion-free $\Gtwo$--structures on a closed $7$--manifold $Y$ is a smooth manifold of dimension $b^3(Y)$ \cite[Theorem C]{Joyce1996}.
The $\Gtwo$ period map $\Pi \co \sT \to \rH_{\dR}^3(Y) \oplus \rH_{\dR}^4(Y)$ defined by
\begin{equation*}
  \Pi\paren*{\phi \cdot \Diff_0(Y)}
  \coloneq
  ([\phi],[\psi])
  \qwithq
  \psi \coloneq *_\phi\phi
\end{equation*}
is a Lagrangian immersion \cites[Lemma 1.1.3]{Joyce1996a}.%
\footnote{%
  Whether or not $\Pi$ is an embedding is an open question.
}
It is constrained by the following inequalities \cites[Lemma 1.1.2]{Joyce1996a}[§IV.2.A and §IV.2.B]{Harvey1982}:
\begin{enumerate}
\item
  \label{It_Topological}
  ${\displaystyle \int_Y \alpha \wedge \alpha \wedge \phi < 0}$
  ~
  for every non-zero $[\alpha] \in \rH_{\dR}^2(Y)$ if $\pi_1(Y)$ is finite.
\item
  \label{It_Instanton}
  ${\displaystyle \int_Y p_1(V) \wedge \phi = -\tfrac{1}{4\pi^2} \YM(A) < 0}$
  ~
  for every vector bundle $V$ which admits a non-flat $\Gtwo$--instanton $A$;
  in particular, for $V = TY$ unless $Y$ is covered by $T^7$.
\item
  \label{It_Associative}
  ${\displaystyle \int_P \phi = \vol(P) > 0}$
  ~
  for every associative submanifold $P \imm Y$.
\item
  \label{It_Coassociative}
  ${\displaystyle \int_Q \psi = \vol(Q) > 0}$
  ~
  for every coassociative submanifold $Q \imm Y$.
\end{enumerate}
These should be compared with the inequalities cutting out the Kähler cone of a Calabi--Yau $3$--fold.

By analogy with Calabi--Yau $3$--folds,
\citet[§3]{Halverson2015} suggest that the above inequalities completely characterise the ideal boundary of $\sT(Y)$.
Of course, making this precise is complicated by the fact that the notions of $\Gtwo$--instanton and (co)associative submanifold depend on the $\Gtwo$--structure $\phi$.
The situation would be improved if there were invariants whose non-vanishing guaranteed the existence of $\Gtwo$--instantons and (co)associative submanifolds as suggested by \citet[§3]{Donaldson1998}.
However, their construction is fraught with enormous difficulty \cite{Donaldson2009,Joyce2016,Haydys2017,Walpuski2013a,Doan2017d}.

A more down to earth problem is to exhibit concrete examples of degenerating families of $\Gtwo$--manifolds which admit $\Gtwo$--instantons whose Yang--Mills energies tend to zero \cite{Walpuski2011} or which admit (co)associative submanifolds whose volumes tend to zero.
The purpose of this article is to present examples of the latter in  $\Gtwo$--manifolds arising from Joyce's generalised Kummer construction.
Although these examples had been anticipated (e.g. by \citet[§6.2]{Halverson2015}),
their rigorous construction has only recently become possibly due to the work of \citet{Platt2020}.

\begin{remark}
  Of course,
  there are already numerous examples of closed associative submanifolds in the literature.
  \begin{enumerate}
  \item
    \citeauthor{Joyce1996a} \cites[§4.2]{Joyce1996a}[§12.6]{Joyce2000} has constructed (co)associative submanifolds in generalised Kummer constructions as fixed-point sets of involutions.
  \item
    \citet[§5.5 and §7.2.2]{Corti2012a} have constructed associative submanifolds in twisted connected sums using rigid holomorphic curves and special Lagrangians in asymptotically cylindrical Calabi--Yau 3--folds.
  \item
    In the physics literature,
    \citet[§4.4]{Braun2018} have proposed a construction of infinitely many associative submanifolds in certain twisted connected sums.
    An important ingredient in the proof of this conjecture will be a gluing theorem for associative submanifolds in twisted connected sums analogous to \cite{SaEarp2013}.
    Building on \cite{Braun2018},
    \citet[§2.2 and §4.2]{Acharya2019} have constructed infinitely many associative submanifolds in certain $\Gtwo$--orbifolds (without using any analytic methods).
  \item
    \citet{Lotay2014,Kawai2015,Ball2000} have produced a wealth of examples of associative submanifolds in $S^7$,
    the squashed $S^7$, and
    the Berger space with their nearly parallel $\Gtwo$--structures.
  \end{enumerate}
  The novelty of the examples discussed in the present article is that their volumes tend to zero as the ambient $\Gtwo$--manifolds degenerate.
\end{remark}

\paragraph{Acknowledgements.}

This material is based upon work supported by
\href{https://sites.duke.edu/scshgap/}{the Simons Collaboration ``Special Holonomy in Geometry, Analysis, and Physics''} (DP, TW).


\section{Joyce's generalised Kummer construction}
\label{Sec_GeneralisedKummer}

The generalised Kummer construction is a method to produce $\Gtwo$--manifolds by desingularising certain closed flat $\Gtwo$--orbifolds $(Y_0,\phi_0)$ introduced by \citet{Joyce1996,Joyce1996a}.
Besides a rather delicate singular perturbation theory it relies on the fact that the hyperkähler $4$--orbifolds $\H/\Gamma$,
obtained as quotients of the quaternions $\H$ by a finite subgroup $\Gamma < \Sp(1)$,
can be desingularised by hyperkähler $4$--manifolds.
The following model spaces feature prominently throughout this article.

\begin{example}[model spaces]
  \label{Ex_Model}
  Let $X$ be a hyperkähler $4$--orbifold with hyperkähler form
  \begin{equation*}
    \bomega \in (\Im\H)^* \otimes \Omega^2(X).
  \end{equation*}
  Denote by $\vol \in \Omega^3(\Im\H)$ and $\one \in \Omega^1(\Im\H) \otimes \Im\H$ the volume form and the tautological $1$--form respectively.
  \begin{enumerate}
  \item
    \label{Ex_Model_0}
    The $3$--form
    \begin{equation}
      \label{Eq_Phi}
      \vol - \Inner{\one \wedge \bomega} \in \Omega^3(\Im\H \times X)
    \end{equation}
    defines a torsion-free $\Gtwo$--structure on $\Im\H \times X$.
    The corresponding Riemannian metric and the cross-product on $\Im\H \times X$ recover the Riemannian metric and the hypercomplex structure $\bI \in (\Im\H)^* \otimes \Gamma(\End(TX))$ on $X$.
  \item
    \label{Ex_Model_/G}
    Let $G < \SO(\Im\H) \ltimes \Im\H$ be a Bieberbach group; that is: discrete, cocompact, and torsion-free.
    Let $\rho \co G \to \Isom(X)$ be a homomorphism.
    Suppose that $\bomega$ is $G$--invariant;
    that is: for every $(R,t) \in G$
    \begin{equation*}
      \paren[\big]{R^* \otimes \rho(R,t)^*} \bomega = \bomega.
    \end{equation*}
    Set
    \begin{equation*}
      Y \coloneq (\Im \H \times X)/G.
    \end{equation*}
    The $\Gtwo$--structure \autoref{Eq_Phi} descends to a $\Gtwo$--structure
    \begin{equation*}
      \phi \in \Omega^3(Y).
    \end{equation*}
    The canonical projection $p \co Y \to B \coloneq \Im\H/G$ is a flat fibre bundle whose fibres are coassociative and diffeomorphic to $X$;
    cf.~\cite[§3.4]{Barbosa2019}.
    \qedhere
  \end{enumerate}
\end{example}

\begin{remark}[Classification of Bieberbach groups]
  \label{Rmk_Bieberbach}
  If $G < \SO(\Im\H)\ltimes \Im\H$ is a Bieberbach group,
  then $\Lambda \coloneq G \cap \Im\H < \Im \H$ is a lattice and $H \coloneq G/\Lambda < \SO(\Lambda) \times (\Im\H/\Lambda)$ is isomorphic to either $\one$, $C_2$, $C_3$, $C_4$, $C_6$, or $C_2^2$; cf.~\cites{Hantzsche1935}{Conway2003}[§3.3]{Szczepanski2012}.
  More precisely, $G$ is among the following:
  \begin{enumerate}
  \item[($\one$)]
    \label{Rmk_Bieberbach_1}
    $\Lambda$ is arbitrary and $G = \Lambda$.
  \item[($C_2$)]
    \label{Rmk_Bieberbach_C2}
    $\Lambda = \Span{\lambda_1,\lambda_2,\lambda_3}$ with
    \begin{equation*}
      \Inner{\lambda_1,\lambda_2} = \Inner{\lambda_1,\lambda_3} = 0.
    \end{equation*}
    $G$ is generated by $\Lambda$ and $(R_2,\frac12\lambda_1)$ with $R_2 \in \SO(\Lambda)$ as in \autoref{Eq_R}.
  \item[($C_3$)]
    \label{Rmk_Bieberbach_C3}
    $\Lambda = \Inner{\lambda_1,\lambda_3,\lambda_3}$ with
    \begin{equation}
      \label{Eq_Bieberbach_C3}
      \Inner{\lambda_1,\lambda_2} = \Inner{\lambda_1,\lambda_3} = 0 \qandq
      \abs{\lambda_2}^2 = \abs{\lambda_3}^2 = -2\Inner{\lambda_2,\lambda_3}.
    \end{equation}
    $G$ is generated by $\Lambda$ and $(R_3,\frac13\lambda_1)$ with $R_3 \in \SO(\Lambda)$ as in \autoref{Eq_R}.
  \item[($C_4$)]
    \label{Rmk_Bieberbach_C4}
    $\Lambda = \Inner{\lambda_1,\lambda_3,\lambda_3}$ with
    \begin{equation*}
      \Inner{\lambda_1,\lambda_2} = \Inner{\lambda_1,\lambda_3} = 0, \quad
      \abs{\lambda_2}^2 = \abs{\lambda_3}^2, \qandq
      \Inner{\lambda_2,\lambda_3} = 0.
    \end{equation*}
    $G$ is generated by $\Lambda$ and $(R_4,\frac14\lambda_1)$ with $R_4 \in \SO(\Lambda)$ as in \autoref{Eq_R}.
  \item[($C_6$)]
    \label{Rmk_Bieberbach_C6}
    $\Lambda = \Inner{\lambda_1,\lambda_3,\lambda_3}$ with \autoref{Eq_Bieberbach_C3}.
    $G$ is generated by $\Lambda$ and $(R_6,\frac16\lambda_1)$ with $R_6 \in \SO(\Lambda)$ as in \autoref{Eq_R}.
  \item[($C_2^2$)]
    \label{Rmk_Bieberbach_C2C2}
    $\Lambda = \Inner{\lambda_1,\lambda_2,\lambda_3}$ with
    \begin{equation*}
      \Inner{\lambda_1,\lambda_2} = \Inner{\lambda_2,\lambda_3} = \Inner{\lambda_3,\lambda_1} = 0.
    \end{equation*}
    $G$ is generated by $\Lambda$, $(R_+,\frac12(\lambda_1+\lambda_2))$, and $(R_-,\frac12(\lambda_2+\lambda_3))$ with $R_\pm \in \SO(\Lambda)$ as in \autoref{Eq_R}.
  \end{enumerate}
  Here $R_2,R_3,R_4,R_6,R_\pm \in \GL_3(\Z)$ are defined by
  \begin{equation}
    \label{Eq_R}
    \begin{gathered}
      R_2 \coloneq
      \begin{pmatrix}
        1 & 0 & 0 \\
        0 & -1 & 0\\
        0 & 0 & -1        
      \end{pmatrix}, \quad
      R_3
      \coloneq
      \begin{pmatrix}
        1 & 0 & 0\\
        0 & -1 & 1 \\
        0 & -1 & 0        
      \end{pmatrix}, \quad
      R_4
      \coloneq
      \begin{pmatrix}
        1 & 0 & 0\\
        0 & 0 & 1 \\
        0 & -1 & 0        
      \end{pmatrix}, \\
      R_6
      \coloneq
      \begin{pmatrix}
        1 & 0 & 0\\
        0 & 1 & -1 \\
        0 & 1 & 0        
      \end{pmatrix}, \qandq
      R_\pm
      \coloneq
      \begin{pmatrix}
        \pm 1 & 0 & 0 \\
        0 & \mp 1 & 0 \\
        0 & 0 & -1
      \end{pmatrix}.
    \end{gathered}
  \end{equation}
  $\GL_3(\Z)$ is identified with $\GL(\Lambda)$ by the choice of generators of $\Lambda$.
\end{remark}

The generalised Kummer construction involves a choice of the following data.

\begin{definition}
  \label{Def_ResolutionData}
  Let $(Y_0,\phi_0)$ be a flat $\Gtwo$--orbifold.
  Denote the connected components of the singular set of $Y_0$ by $S_\alpha$ ($\alpha \in A$).
  \defined{Resolution data}
  $\fR = (\Gamma_\alpha,G_\alpha,\rho_\alpha;R_\alpha,\jmath_\alpha;\hat X_\alpha,\hat\bomega_\alpha,\hat\rho_\alpha,\tau_\alpha)_{\alpha \in A}$
  for $(Y_0,\phi_0)$ consist of the following for every $\alpha \in A$:
  \begin{enumerate}
  \item
    \label{Def_ResolutionData_AlgebraicData}
    A finite subgroup $\Gamma_\alpha < \Sp(1)$, a Bieberbach group $G_\alpha < \SO(\Im\H)\ltimes \Im\H$, and a homomorphism $\rho_\alpha \co G_\alpha \to N_{\SO(\H)}(\Gamma_\alpha) \incl \Isom(\H/\Gamma_\alpha)$ as in \autoref{Ex_Model}~\autoref{Ex_Model_/G} with $X \coloneq \H/\Gamma_\alpha$ and its canonical hyperkähler form $\bomega$.
  \item
    \label{Def_ResolutionData_Model}        
    A number $R_\alpha >0$ defining the open set
    \begin{equation*}
      U_\alpha \coloneq \left.\paren[\big]{\Im \H \times \paren{B_{2R_\alpha}(0)/\Gamma_\alpha}}\right/G_\alpha \subset Y_\alpha
    \end{equation*}
    and an open embedding $\jmath_\alpha \co U_\alpha \to Y_0$ 
    satisfying $S_\alpha \subset \im \jmath_\alpha$ and
    \begin{equation*}
      \jmath_\alpha^*\phi_0 = \phi_{\alpha}
    \end{equation*}
    with $(Y_\alpha,\phi_\alpha)$ denoting the model space associated with $\H/\Gamma_\alpha$, $\bomega$, $G_\alpha$, and $\rho_\alpha$.
  \item
    \label{Def_ResolutionData_Resolution}
    A hyperkähler $4$--manifold $\hat X_\alpha$ with hyperkähler form $\hat\bomega_\alpha \in (\Im\H)^* \otimes \Omega^2(\hat X_\alpha)$,
    a homomorphism $\hat \rho_\alpha \co G_\alpha \to \Diff(\hat X_\alpha)$ with respect to which $\hat\bomega_\alpha$ is $G_\alpha$--invariant (in the sense of \autoref{Ex_Model}~\autoref{Ex_Model_/G}),
    a compact subset $K_\alpha \subset \hat X_\alpha$, and
    a $G_\alpha$--equivariant open embedding $\tau_\alpha \co \hat X_\alpha \setminus K_\alpha \to \H/\Gamma_\alpha$ with
    $\paren{\H\setminus B_{R_\alpha}(0)}/\Gamma \subset \im\tau_\alpha$ and    
    \begin{equation}
      \label{Def_ResolutionData_Resolution_Asymptotic}
      \abs{\nabla^k\paren{\tau_*\hat\bomega_\alpha-\bomega}} = O\paren{r^{-4-k}}
    \end{equation}
    for every $k \in \N_0$.
    \qedhere
  \end{enumerate}
\end{definition}

\begin{remark}[ADE classification of finite subgroups of $\Sp(1)$]
  \label{Rmk_FiniteSubgroupsSp1_ADE}
  \citet{Klein1884} classified the (non-trivial) finite subgroups $\Gamma < \Sp(1)$.
  They obey an ADE classification.
  $\Gamma$ is isomorphic to either:
  \begin{enumerate}
  \item[($A_k$)] 
    a cyclic group $C_{k+1}$,    
  \item[$(D_k)$]
    a dicyclic group $\Dic_{k-2}$,
  \item[($E_6$)]
    the binary tetrahedral group $2T$,
  \item[($E_7$)]
    the binary octahedral group $2O$,
    or
  \item[($E_8$)]
    the binary icosahedral group $2I$.
    \qedhere
  \end{enumerate}
\end{remark}

\begin{remark}
  Whether or not the data in \autoref{Def_ResolutionData}~\autoref{Def_ResolutionData_AlgebraicData} and \autoref{Def_ResolutionData_Model} exists is a property of a neighborhood of the singular set of $Y_0$.
  If it does exist, then it is essentially unique.
  The data in \autoref{Def_ResolutionData}~\autoref{Def_ResolutionData_Resolution} involves a choice.
\end{remark}

\begin{remark}
  There are many examples of closed flat $\Gtwo$--orbifolds admitting resolution data in the above sense; see \cites[§3]{Joyce1996a}[§12]{Joyce2000}[§3]{Barrett2006}[§5.3.4 and §5.3.5]{Reidegeld2017}.
  They arise from certain crystallographic groups $G < \Gtwo \ltimes \R^7$.
  It would be interesting to classify these (possibly computer-aided) to grasp the full scope of Joyce's generalised Kummer construction.
  Partial results have been obtained by \citet[§3.2]{Barrett2006}, and
  \citet[Theorem 5.3.1]{Reidegeld2017} observed that in \autoref{Def_ResolutionData}~\autoref{Def_ResolutionData_AlgebraicData} precisely $C_2$, $C_3$, $C_4$, $C_6$, $\Dic_2$, $\Dic_3$, and $2T$ can appear.
\end{remark}

\begin{remark}[scaling resolution data]
  For every $(t_\alpha) \in (0,1]^A$ the data $\hat\bomega_\alpha$ and $\tau_\alpha$ in \autoref{Def_ResolutionData}~\autoref{Def_ResolutionData_Resolution} can be replaced with $t_\alpha^2\hat\bomega_\alpha$ and $t_\alpha\tau_\alpha$.
\end{remark}

The following two remarks help to find resolution data $\fR$ with certain properties.

\begin{remark}[Gibbons--Hawking construction of $A_k$ ALE spaces]
  \label{Rmk_AkALE}
  Let $k \in \N$.
  Consider the subgroup $C_k \incl \Sp(1)$ generated by right multiplication with $e^{2\pi i/k}$.
  (Of course, $i$ can be replaced by $\hat\xi \in S^2 \subset \Im\H$ throughout.)
  The $A_k$ ALE hyperkähler $4$--manifolds used to resolve $\H/C_k$ can be understood concretely using the Gibbons--Hawking construction \cites{Gibbons1979}[§3.5]{Gibbons1997}.
  \begin{enumerate}
  \item
    \label{Rmk_AkALE_Construction}
    Let
    \begin{equation*}
      \bzeta
      \in \Delta
      \coloneq \Sym_0^k (\Im\H)
      \coloneq \set[\big]{
        [\zeta_1,\ldots,\zeta_k] \in (\Im\H^k)/S_k
        :
        \zeta_1+\cdots+\zeta_k = 0        
      }.
    \end{equation*}    
    Set $Z \coloneq \set{\zeta_1,\ldots,\zeta_k}$ and $B \coloneq \Im\H \setminus Z$.
    The function $V_\bzeta \in C^\infty(B)$ defined by
    \begin{equation*}
      V_\bzeta(q) \coloneq\sum_{a=1}^k \frac{1}{2\abs{q-\zeta_a}}
    \end{equation*}
    is harmonic and
    \begin{equation*}
      [*\rd V_\bzeta] \in \im\paren*{\rH^2(B,2\pi \Z) \to \rH_{\dR}^2(B)}.
    \end{equation*}
    Therefore, there is a $\U(1)$--principal bundle $p_\bzeta \co X_\bzeta^\circ \to B$ and a connection $1$--form $i\theta_\bzeta \in \Omega^1(X_\bzeta^\circ,i\R)$ with
    \begin{equation}
      \label{Eq_GibbonsHawking}
      \rd\theta_\bzeta = -p_\bzeta^*(*\rd V_\bzeta).
    \end{equation}
    Indeed, $p_\zeta$ is determined by $V_\bzeta$ up to isomorphism.
    The Euclidean inner product on $\Im\H$ defines
    \begin{equation*}
      \sigma \in (\Im\H)^* \otimes \Omega^1(\Im\H).
    \end{equation*}
    $X_\bzeta^\circ$ is an incomplete hyperkähler manifold with hyperkähler form $\bomega_\zeta$ defined by
    \begin{equation*}
      \bomega_\zeta \coloneq \theta_\bzeta \wedge p_\bzeta^*\sigma + p_\bzeta^*(V_\bzeta\cdot *\sigma).
    \end{equation*}
  \item
    \label{Rmk_AkALE_Decay}
    The map $p_{\bm{0}} \co (\H\setminus\set{0})/\Gamma \to B$ defined by
    \begin{equation*}
      p_{\bm{0}}([x]) \coloneq \frac{xix^*}{2k}
    \end{equation*}
    is a $\U(1)$--principal bundle with  $[x]\cdot e^{i\alpha} \coloneq [x e^{i\alpha/k}]$.
    The connection $1$--form $i\theta_{\bm{0}}$ defined by
    \begin{equation*}
      \theta_{\bm{0}}([x,v]) \coloneq \frac{\Inner{xi,v}}{k\abs{x}^2}
    \end{equation*}
    satisfies \autoref{Eq_GibbonsHawking}.
    Therefore, $X_{\bm{0}}^\circ = (\H\setminus\set{0})/\Gamma$.
    A straightforward (but slightly tedious) computation reveals that $\bomega_{\bm{0}}$ agrees with the standard hyperkähler form on $(\H\setminus\set{0})/\Gamma$.
    As a consequence,
    $X_\bzeta^\circ$ can be extended to a complete hyperkähler orbifold $X_\bzeta$ by adding $\#Z$ points.
    If
    \begin{equation*}
      \bzeta \in \Delta^\circ
      \coloneq
      \set[\big]{ [\zeta_1,\ldots,\zeta_k] \in \Delta : \zeta_1,\ldots,\zeta_k ~\textnormal{are pairwise distinct} },
    \end{equation*}
    then $X_\bzeta$ is a manifold.
    Since
    \begin{equation}
      \abs{\nabla^k (V_\bzeta - V_{\bm{0}})\circ p_{\bm{0}}} = O\paren{\abs{x}^{-4-k}}
    \end{equation}
    for every $k \in \N_0$,
    the asymptotic decay condition \autoref{Def_ResolutionData_Resolution_Asymptotic} holds.
  \item
    \label{Rmk_AkALE_HolomorphicCurves}
    Let $\bzeta \in \Delta^\circ$.
    If
    \begin{equation*}
      \ell = \set[\big]{ \hat\xi t + \eta: t \in [a,b] } \subset \Im\H
    \end{equation*}
    for some $\eta \in \Im\H$, $[a,b] \subset \R$, and $\hat\xi \in S^2 \subset \Im\H$ is a segment satisfying $\del \ell \subset Z$ and $\ell^\circ \subset B$,
    then
    \begin{equation*}
      \Sigma_\ell \coloneq p_\bzeta^{-1}(\ell) \subset X_\bzeta
    \end{equation*}
    is $I_{\bzeta,\hat\xi}$--holomorphic with
    \begin{equation*}
      I_{\bzeta,\hat\xi} \coloneq \Inner{\bI_\zeta,\hat\xi}
    \end{equation*}
    and $\Sigma_\ell \iso S^2$.
    $\rH_2(X_\bzeta,\Z)$ is generated by the homology classes of these curves.
    In fact, $X_\bzeta$ retracts to a tree of these curves.
  \item
    \label{Rmk_AkALE_Symmetry}
    Define $\Lambda^+ \co \SO(\H) \to \SO(\Im\H)$ by
    \begin{equation*}
      \Lambda^+R\paren{\rd q \wedge \rd \bar q}
      \coloneq
      R_*\paren{\rd q \wedge \rd \bar q}.
    \end{equation*}
    Let $\bzeta \in \Delta^\circ$.
    If $R \in N_{\SO(\H)}(\Gamma)$ satisfies $\Lambda^+R(\bzeta) = \bzeta$,
    then it lifts to an isometry $\hat R \in \Diff(X_\bzeta)$ satisfying 
    \begin{equation*}
      \paren[\big]{(\Lambda^+R)^* \otimes \hat R^*}\bomega_\bzeta = \bomega_\bzeta
      \qandq
      \hat R(\Sigma_\ell) = \Sigma_\ell
    \end{equation*}
    for every $\ell$ as in \autoref{Rmk_AkALE_HolomorphicCurves}.
   \qedhere
  \end{enumerate}
\end{remark}

\begin{remark}[Kronheimer's construction of ALE spaces]
  \label{Rmk_Kronheimer_ALE}
  Let $\Gamma < \Sp(1)$ be a finite subgroup---not necessarily cyclic.
  The ALE hyperkähler $4$--manifolds asymptotic to $\H/\Gamma$ can be understood using the work of 
  \citet{Kronheimer1989,Kronheimer1989a}.
  This is rather more involved that \autoref{Rmk_AkALE} and summarised in the following.
  (This is only used for \autoref{Ex_ExistenceOfAssociatives_Dic2} and might be skipped at the reader's discretion.)
  \begin{enumerate}
  \item
    \label{Rmk_Kronheimer_ALE_Construction}
    Denote by $R \coloneq \C[\Gamma] = \Map(\Gamma,\C)$ the regular representation of $\Gamma$ equipped with the standard $\Gamma$--invariant Hermitian inner product.
    Set
    \begin{equation*}
      S \coloneq \paren{\H \otimes_\R \fu(R)}^\Gamma
      \qandq
      G \coloneq \P U(R)^\Gamma. 
    \end{equation*}
    The adjoint action of $G$ on $S$ has a distinguished hyperkähler moment map
    \begin{equation*}
      \mu \co S \to (\Im\H)^* \otimes \fg^*.
    \end{equation*}
    Denote by $\fz^* \subset \fg^*$ the annihilator of $[\fg,\fg]$.
    For every
    \begin{equation*}
      \bzeta \in \tilde\Delta \coloneq (\Im\H)^* \otimes \fz^*
    \end{equation*}
    the hyperkähler quotient
    \begin{equation*}
      X_\bzeta \coloneq S \hkred\!_{\bzeta} G \coloneq \mu^{-1}(\bzeta)/G
    \end{equation*}
    is an ALE hyperkähler $4$--orbifold asymptotic to $\H/\Gamma$.
  \item
    \label{Rmk_Kronheimer_ALE_Smooth}
    Set
    \begin{equation*}
      \Pi \coloneq \set[\big]{ i\pi \in \fu(R)^\Gamma\setminus\set{0,\one} : \pi^2 = \pi }.
    \end{equation*}
    (This set is in bijection with the set of non-trivial proper subrepresentations of $R$.)
    For $i\pi \in \Pi$ denote by $D_{i\pi} \coloneq [i\pi]^0 \subset \fz^*$ the annihilator of $[i\pi] \in \fg$.
    If
    \begin{equation*}
      \bzeta \in \tilde\Delta^\circ \coloneq \tilde\Delta \setminus D \qwithq
      D \coloneq \bigcup_{i\pi \in \Pi} (\Im \H)^* \otimes D_{i\pi},
    \end{equation*}
    then $X_\bzeta$ is a manifold.
  \item
    \label{Rmk_Kronheimer_ALE_McKay}    
    \autoref{Rmk_FiniteSubgroupsSp1_ADE} associates a Dynkin diagram with $\Gamma$.
    According to the McKay correspondence \cite{McKay1980},    
    the non-trivial irreducible representations of $\Gamma$ correspond to the vertices of this diagram.
    The corresponding root system $\Phi$ has a preferred set of positive roots $\Phi^+$.
    The latter can be identified with $\Pi$.
    In particular, the hyperplanes $D_{i\pi}$ correspond to the walls of the Weyl chambers of $\Phi$.
  \item
    \label{Rmk_Kronheimer_ALE_HolomorphicCurves}
    Let $\bzeta \in \tilde\Delta^\circ$.
    Let $\alpha \in \Phi$ be a simple root.
    Define $\xi \in \Im\H$ by $\Inner{\xi,\cdot} \coloneq \bzeta(\alpha) $ and set $\hat \xi \coloneq \xi/\abs{\xi}$.
    There is a $I_{\bzeta,\hat\xi}$--holomorphic curve
    \begin{equation*}
      \Sigma_\alpha \subset X_\bzeta
    \end{equation*}
    with $\Sigma_\alpha \iso S^2$.
    $\rH_2(X_\bzeta)$ is generated by the homology classes of these curves.
    In fact, $X_\bzeta$ retracts to a tree of these curves.
    This identifies $\rH_2(X_\bzeta)$ with the root lattice $\Z\Phi$.
  \item
    \label{Rmk_Kronheimer_ALE_Symmetry}
    Let $\bzeta \in \tilde\Delta^\circ$.
    If $R \in N_{\SO(\H)}(\Gamma)$ satisfies $\Lambda^+R(\bzeta) = \bzeta$,
    then it lifts to an isometry $\hat R \in \Diff(X_\bzeta)$ satisfying 
    \begin{equation*}
      \paren{R^* \otimes \hat R^*}\bomega_\bzeta = \bomega_\bzeta \qandq
      \hat R(\Sigma_\alpha) = \Sigma_\alpha.
    \end{equation*}
    Denote by $W$ the Weyl group of $\Phi$.
    Every $\sigma \in W$ induces a hyperkähler isometry $\hat \sigma \co X_\bzeta \iso X_{\sigma(\bzeta)}$ satisfying $\hat \sigma(\Sigma_\alpha) = \Sigma_{\sigma(\alpha)}$.
    In particular, $\tilde\Delta$ and $\tilde\Delta^\circ$ can be replaced with
    \begin{equation*}
      \Delta \coloneq \tilde\Delta/W \qandq
      \Delta^\circ \coloneq \tilde\Delta^\circ/W.      
    \end{equation*}
  \end{enumerate}
  Of course, for $\Gamma = C_k$ the above parallels \autoref{Rmk_AkALE}.
\end{remark}

The generalised Kummer construction proceeds by constructing an approximate resolution and correcting it via singular perturbation theory.

\begin{definition}[approximate resolution]
  \label{Def_Approximate_Resolution}
  Let $(Y_0,\phi_0)$ be a flat $\Gtwo$--orbifold together with resolution data $\fR$.
  Let $t \in (0,1]$.
  Set
  \begin{equation*}
    Y_0^\circ \coloneq Y_0 \setminus \bigcup_{\alpha \in A} \jmath_\alpha\paren[\Big]{\paren[\big]{\Im \H \times \paren{\overline B_{R_\alpha}(0)/\Gamma_\alpha}}\big{/}G_\alpha}.
  \end{equation*}
  For $\alpha \in A$ denote by $(\hat Y_{\alpha,t},\hat\phi_{\alpha,t})$ the model space associated with $\hat X_\alpha$, $t^2\hat\bomega_\alpha$, $G_\alpha$, and $\hat\rho_\alpha$.
  Set
  \begin{align*}
    \hat Y_t^\circ
    &\coloneq
      \coprod_{\alpha \in A}
      \hat Y_{\alpha,t}^\circ
      \qwithq
      \hat Y_{\alpha,t}^\circ
      \coloneq
      \left.\paren[\Big]{\Im \H \times \paren[\big]{K_\alpha \cup (t\tau_\alpha)^{-1}\paren{B_{2R_\alpha}(0)/\Gamma_\alpha}}}\right/G_\alpha, \\
    \hat V_t
    &\coloneq
      \coprod_{\alpha \in A}
      \hat V_{\alpha,t}
      \qwithq
      \hat V_{\alpha,t}
      \coloneq
      \left.\paren[\Big]{\Im \H \times (t\tau_\alpha)^{-1}\paren[\big]{\paren{B_{2R_\alpha}(0)\setminus B_{R_\alpha}(0)}/\Gamma_\alpha}}\right/G_\alpha, \qand \\
    V
    &\coloneq
      \coprod_{\alpha \in A}
      V_{\alpha\phantom{,t}}
      \qwithq
      V_{\alpha\phantom{,t}}
      \coloneq
      \left.\paren[\Big]{\Im \H \times \paren[\big]{\paren{B_{2R_\alpha}(0)\setminus B_{R_\alpha}(0)}/\Gamma_\alpha}}\right/G_\alpha.
  \end{align*}
  Denote by $f \co \hat V_t \to V$ the diffeomorphism induced by $\jmath_\alpha$ and $t\tau_\alpha$ ($\alpha \in A$).
  Denote by $Y_t$ the $7$--manifold obtained by gluing $\hat Y_t^\circ$ and $Y_0^\circ$ along $f$:
  \begin{equation*}
    Y_t \coloneq \hat Y_t^\circ \cup_f Y_0^\circ.
  \end{equation*}
  A cut-and-paste procedure (whose details are swept under the rug \cites[Proof of Theorem 2.2.1]{Joyce1996a}[§11.5.3]{Joyce2000}) produces a closed $3$--form
  \begin{equation*}
    \tilde\phi_t \in \Omega^3(Y_t)
  \end{equation*}
  which agrees with $\hat \phi_{\alpha,t}$ on $\hat Y_\alpha^\circ \setminus \hat V_{\alpha,t}$ ($\alpha \in A$) and with $\phi_0$ on $Y_0^\circ \setminus V$;
  moreover:
  if $t$ is sufficiently small,
  then $\tilde\phi_t$ defines a $\Gtwo$--structure on $Y_t$.
\end{definition}

\begin{remark}
  \label{Rmk_HomologyMaps}
  Since $\hat X_\alpha$ retracts to a compact subset,
  there are canonical maps
  \begin{equation*}
    \eta_\alpha \co H_\bullet(\hat Y_{\alpha,t},\Z) \iso H_\bullet(\hat Y_{\alpha,t}^\circ,\Z) \to H_\bullet(Y_t,\Z).
    \qedhere
  \end{equation*}
\end{remark}

\begin{remark}
  \label{Rmk_T-3Phi}
  As $t$ tends to zero, the Riemannian metric $\tilde g_t$ associated with $\tilde\phi_t$ degenerates quite severely:
  $\Abs{R_{\tilde g_t}}_{L^\infty} \sim t^{-2}$ and $\inj(\tilde g_t) \sim t^{-1}$.
  To ameliorate this it can be convenient to pass to the Riemannian metric $t^{-2}\tilde g_t$ associated with $t^{-3}\tilde\phi_t$.
\end{remark}

The following refinement of Joyce's existence theorem for torsion-free $\Gtwo$--structures 
\cites[Theorem B]{Joyce1996}[Theorems G1 and G2]{Joyce2000} is crucial.

\begin{theorem}[{\citet[Corollary 4.31]{Platt2020}}]
  \label{Thm_PerturbG2Structure}
  Let $\fR$ be resolution data for a closed flat $\Gtwo$--orbifold $(Y_0,\phi_0)$.
  Let $\alpha \in (0,1/16)$.
  There are $T_0=T_0(\fR),c=c(\fR,\alpha) > 0$ and for every $t \in (0,T_0)$ there is a torsion-free $\Gtwo$--structure $\phi_t \in \Omega^3(Y_t)$ with $[\phi_t] = [\tilde\phi_t] \in \rH_{\dR}^3(Y_t)$ satisfying
  \begin{equation*}
    \Abs{t^{-3}(\phi_t - \tilde \phi_t)}_{C^{1,\alpha}} \leq c t^{3/2-\alpha}.
  \end{equation*}  
  Here $\Abs{-}_{C^{1,\alpha}}$ is with respect to
  $t^{-2}\tilde g_t$.
\end{theorem}

\begin{remark}[$K$--equivariant generalised Kummer construction]
  \label{Rmk_GeneralisedKummer_KInvariant}
  Let $(Y_0,\phi_0)$ be a closed flat $\Gtwo$--orbifold.
  Let $K$ be a group.
  Let $\lambda \co K \to \Diff(Y_0)$ be a homomorphism with respect to which $\phi_0$ is $K$--invariant.
  $K$ acts on the singular set of $Y_0$ and, therefore, on $A$.
  \defined{$K$--equivariant resolution data} for $(Y_0,\phi_0;\lambda)$ consist of resolution data $\fR = (\Gamma_\alpha,G_\alpha,\rho_\alpha;R_\alpha,\jmath_\alpha;\hat X_\alpha,\hat\bomega_\alpha,\hat\rho_\alpha,\tau_\alpha)_{\alpha\in A}$ for $(Y_0,\phi_0)$ with the property that for every $\alpha \in A$ and $g \in K$
  \begin{equation*}
    \Gamma_{g\alpha} = \Gamma_\alpha, \quad
    G_{g\alpha} = G_\alpha, \quad
    \rho_{g\alpha} = \rho_\alpha, \qandq
    R_{g\alpha} = R_\alpha
  \end{equation*}  
  and of the following additional data for every $\alpha \in A$:
  \begin{enumerate}
  \item
    A pair of homomorphisms $\lambda_\alpha \co K \to N_{\SO(\Im\H)\ltimes \Im\H}(G_\alpha) < \SO(\Im\H)\ltimes \Im \H$ and
    $\kappa_\alpha \co K \to N_{N_{\SO(\H)}(G_\alpha)}(\rho_\alpha(G_\alpha)) \incl \Isom(\H/\Gamma_\alpha)$
    such that for every $g \in G$
    \begin{equation*}
      \lambda(g) \circ \jmath_\alpha = \jmath_{g\alpha} \circ \sqparen{\lambda_\alpha(g) \times \kappa_\alpha(g)}.
    \end{equation*}
    Here $\sqparen{\lambda_\alpha(g) \times \kappa_\alpha(g)}$ denotes the induced isometry of $U_\alpha = U_{g\alpha}$.
  \item
    A homomorphism $\hat\kappa_\alpha \co K \to N_{\Diff(\hat X_\alpha)}(\hat\rho_\alpha(G_\alpha))$
    such that $\hat\bomega_\alpha$ is $K$--invariant with respect to $\lambda_\alpha$ and $\hat\kappa_\alpha$ (in the sense of \autoref{Ex_Model}~\autoref{Ex_Model_/G}) and $\tau_\alpha$ is $K$--equivariant with respect to $\hat\kappa_\alpha$.
  \end{enumerate}  
  The approximate resolution in can \autoref{Def_Approximate_Resolution} be done so that $\lambda$ and $(\lambda_\alpha,\kappa_\alpha)_{\alpha \in A}$ lift to a homomorphism $\lambda_t \co K \to \Diff(Y_t)$ with resepect to which $\tilde\phi_t$ is $K$--invariant.
  In this situation,
  $\tilde\phi_t$ construted by \autoref{Thm_PerturbG2Structure} is $K$--invariant.
\end{remark}


\section{Perturbing Morse--Bott families of associative submanifolds}
\label{Sec_PerturbingMorseBottFamilies}

Throughout,
let $Y$ be a $7$--manifold with a $\Gtwo$--structure $\phi \in \Omega^3(Y)$.
Set
\begin{equation*}
  \psi \coloneq *\phi \in \Omega^4(Y).
\end{equation*}
Encode the torsion of $\phi$ as the section $\tau \in \Gamma\paren{\gl(TY)}$ defined by
\begin{equation*}
  \nabla_v \psi \eqcolon \tau(v)^\flat \wedge \phi.
\end{equation*}
Here $-^\flat \co TY \to T^*Y$ denotes the isomorphism induced by the Riemannian metric.

\begin{definition}
  A closed oriented $3$--dimensional immersed submanifold $P \imm Y$ is \defined{($\phi$--)associative} if
  \begin{equation*}
    \phi|_P > 0
    \qandq
    (i_v\psi)|_P = 0
    \qforeveryq
    v \in NP
  \end{equation*}
  or, equivalently, if it is $\phi$--semi-calibrated \cite[Theorem 1.6]{Harvey1982}.
\end{definition}

\begin{example}
  \label{Ex_Model_Associative}
  Assume the situation of \autoref{Ex_Model}~\autoref{Ex_Model_/G}.
  Let $\hat\xi \in S^2 \subset \Im\H$, $L > 0$, and $\Sigma \subset X$.
  Suppose that $\Sigma$ is a closed $I_{\hat\xi}$--holomorphic curve with $I_{\hat \xi} \coloneq \Inner{\bI,\hat\xi}$,
  $\xi \coloneq L\hat\xi \in \Lambda < G$ is primitive,
  $\Z\xi < G$ is normal, and,
  for every $g \in G$,
  $\rho(g)(\Sigma) = \Sigma$.
  In this situation,
  for every
  \begin{equation*}
    [\eta] \in
    M/H \qwithq
    M \coloneq \paren{\Im\H/\R\xi}/\paren{\Lambda/\Z\xi} \iso T^2
    \qandq
    H \coloneq G/\Lambda
  \end{equation*}
  the submanifold  
  \begin{equation*}
    P_{[\eta]} \coloneq \left.\paren[\big]{(\R \xi + \eta) \times \Sigma}\right/\Z\xi \imm Y
  \end{equation*}
  is associative and diffeomorphic to the mapping torus $T_\mu$ of $\mu = \rho(\xi) \in \Diff(\Sigma)$.
\end{example}

\begin{remark}
  \label{Rmk_Model_Associative_Possibilities}
  $\Z\xi < G$ is normal if and only if $\xi$ is an eigenvector of every $R \in G \cap \SO(\Im\H)$.
  Direct inspection of \autoref{Rmk_Bieberbach} reveals the following possibilities (without loss of generality):
  \begin{enumerate}
  \item[$(\one)$]
    $H \iso \one$ and $\xi \in \Lambda$ is any primitive element.
  \item[($C_2^+$)]
    $H \iso C_2$ and $\xi = \lambda_1$.
    The orbifold $M/H$ has $4$ singularities:
    each with isotropy $C_2$.
  \item[($C_2^-)$]
    $H \iso C_2$ and $\xi = \lambda_2$.
    $M/H$ is diffeomorphic to the Klein bottle $\RP^2 \# \RP^2$.
  \item[($C_3$)]
    $H \iso C_3$ and $\xi = \lambda_1$.
    The orbifold $M/H$ has $3$ singularities:
    each with isotropy $C_3$.
  \item[($C_4$)]
    $H \iso C_4$ and $\xi = \lambda_1$.
    The orbifold $M/H$ has $3$ singularities:
    two with isotropy $C_4$, one with isotropy $C_2$.
  \item[($C_6$)]
    $H \iso C_6$ and $\xi = \lambda_1$.
    The orbifold $M/H$ has $3$ singularities:
    one with isotropy $C_6$,
    one with isotropy $C_3$,
    one with isotropy $C_2$.
  \item[($C_2^2$)]
    $H \iso C_2^2$ and $\xi = \lambda_1$.
    The orbifold $M/H$ has $2$ singularities:
    each with isotropy $C_2$.
    \qedhere
  \end{enumerate}
\end{remark}

\begin{remark}
  \label{Rmk_Model_Associative_Mu=1}
  The examples discussed in \autoref{Sec_Examples} are based on \autoref{Ex_Model_Associative} with $\mu = \id_\Sigma$.
\end{remark}

Let $\beta \in \rH_3(Y,\Z)$.
Denote by $\sS = \sS(Y)$ the oribfold of closed oriented $3$--dimensional immersed submanifolds $P \imm Y$ with $\phi|_P > 0$ and $[P] = \beta$;
cf.~\cite[§44]{Kriegl1997}.
Define $\delta\Upsilon = \delta\Upsilon^\psi \in \Omega^1(\sS)$ by
\begin{equation*}
  \delta\Upsilon_P(v) \coloneq \int_P i_v\psi
  \qforq
  v \in T_P\sS = \Gamma(NP).
\end{equation*}
By construction, $P \in \sS$ is associative if and only if it is a critical point of $\delta\Upsilon$.

If $\rd\psi = 0$,
then $\delta\Upsilon$ is closed;
indeed:
there is a covering map $\pi \co \tilde\sS \to \sS$ such that $\pi^*\delta\Upsilon$ is exact.
The covering map $\pi$ is the principal covering map associated with the sweep-out homomorphism
\begin{equation*}
  \textnormal{sweep} \co \pi_1(\sS) \to \rH_4(Y).
\end{equation*}
More concretely:
choose $P_0 \in \sS$ and denote by $\tilde S$ the set of equivalence classes $[P,Q]$ of pairs consisting of $P \in \sS$ and a $4$--chain $Q$ satisfying $\del Q = P - P_0$ with respect to the equivalence relation $\sim$ defined by
\begin{equation*}
  (P_1,Q_1) \sim (P_2,Q_2)
  \iff
  \paren[\big]{P_1 = P_2 \qandq
    [Q_1-Q_2] = 0 \in \rH_4(Y,\Z)}.
\end{equation*}
$\tilde\sS$ admits a unique smooth structure such that the canonical projection map $\pi \co \tilde S \to S$ is a smooth covering map.
Evidently, $\Upsilon = \Upsilon^\psi \in C^\infty(\tilde\sS)$ defined by
\begin{equation*}
  \Upsilon([P,Q])
  \coloneq
  \int_Q \psi
\end{equation*}
satisfies
\begin{equation*}
  \rd\Upsilon = \pi^*(\delta\Upsilon).
\end{equation*}

If $\bP \co M \to \sS$ is a smooth map,
then critical points of $\bP^*(\delta\Upsilon)$ need not correspond to associative submanifolds.
However, the following trivial observation turns out to be helpful.

\begin{lemma}
  \label{Lem_Trivial}
  Let $\bP \co M \to \sS$ be a smooth map.
  If for every $x \in M$
  \begin{equation*}
    \ker (\delta\Upsilon)_{\bP(x)} + \im T_x\bP = T_{\bP(x)}\sS,
  \end{equation*}
  then $\bP(x)$ is associative if and only if $x$ is a zero of $\bP^*(\delta\Upsilon)$.
  \qed
\end{lemma}

\begin{remark}
  \label{Rmk_Trivial}
  \autoref{Lem_Trivial} is particularly useful if there is a mechanism that forces $\bP^*(\delta\Upsilon)$ to have zeros;
  e.g.:
  \begin{enumerate}
  \item
    \label{Rmk_Trivial_Euler}
    If $M$ is closed,
    then $\bP^*(\delta\Upsilon)$ has $\chi(M)$ zeros (counted with signs and multiplicities).
  \item
    \label{Rmk_Trivial_H}
    If there is a finite group $H$ acting on $M$ and
    $\bP^*(\delta\Upsilon)$ is $H$--invariant,
    then every isolated fixed-point is a zero.
  \item
    \label{Rmk_Trivial_Exact}
    If $M$ is closed and $\bP^*(\delta\Upsilon)$ is exact,
    then it has at least two zeros (indeed: at least three unless $M$ is homeomorphic to a sphere).
    $\bP^*(\delta\Upsilon)$ is exact if and only if the composite homomorphism
    \begin{equation*}
      \pi_1(M) \xrightarrow{\pi_1(\bP)} \pi_1(\sS) \xrightarrow{\textnormal{sweep}} \rH_4(Y) \xrightarrow{\Inner{-,[\psi]}} \R
    \end{equation*}
    vanishes.
    \qedhere
  \end{enumerate}
\end{remark}

The deformation theory of associative submanifolds is quite well-behaved.
Here is a summary of the salient points.

\begin{definition}
  A \defined{tubular neighborhood} of $P \in \sS$ is an open immersion $\jmath \co U \imm Y$ with $U \subset NP$ an open neighborhood of the zero section in $NP$ satisfying $[0,1]\cdot U = U$.
\end{definition}

Let $\jmath \co U \imm V$ be a tubular neighborhood of $P \in \sS$.
Define $\bQ = \bQ_\jmath \co \Gamma(U) \to \sS$ by
\begin{equation*}
  \bQ(v) \coloneq \jmath(\Gamma_v)
  \qwithq
  \Gamma_v \coloneq \im v \subset NP.
\end{equation*}
This map is (the inverse of) a chart of $\sS$.
Since $\Gamma(U) \subset \Gamma(NP)$ open,
$\Omega^1(\Gamma(U))$ can be identified with $C^\infty\paren{\Gamma(U),\Gamma(NP)^*}$.
Therefore, it makes sense to Taylor expand $\bQ^*(\delta\Upsilon)$.
If $P$ is associative,
then the first order term is independent of $\jmath$.

\begin{definition}
  \label{Def_Associative_Linearisation} 
  Let $P \in \sS$ be associative.
  Define $\gamma \co \Hom(TP,NP) \to NP$ by
  \begin{equation*}
    \Inner{\gamma(v\cdot u^\flat),w} \coloneq \phi(u,v,w).
  \end{equation*}
  Denote by $\tau^\perp \in \Gamma(\gl(NP))$ the restriction of $\tau \in \Gamma(\gl(TY))$.
  The \defined{Fueter operator} $D = D_P \co \Gamma(NP) \to \Gamma(NP)$ associated with $P$ is defined by
  \begin{equation*}
    D \coloneq - \gamma\nabla + \tau^\perp.
    \qedhere
  \end{equation*}
\end{definition}

\begin{prop}[{\citet[§5]{McLean1998}, \citet[Theorem 6]{Akbulut2008}, \citet[Theorem 2.1]{Gayet2014}, \citet[Theorem 2.12]{Joyce2016}}]
  \label{Prop_Associative_Linearisation+Remainder}
  Let $P \in \sS$ be associative.
  Let $\jmath \co U \imm V$ be a tubular neighborhood of $P$.
  There are a constant $c = c(\jmath) > 0$ and a smooth map $\sN = \sN_\jmath \in C^\infty(\Gamma(U),\Gamma(NP))$ such that
  \begin{equation*}
    \Inner{\bQ^*(\delta\Upsilon)(v),w}
    =
    \Inner{Dv + \sN(v),w}_{L^2}.
  \end{equation*}
  and
  \begin{equation*}
    \Abs{\sN(v)-\sN(w)}_{C^{0,\alpha}}
    \leq
    c \paren*{\Abs{v}_{C^{1,\alpha}}+\Abs{w}_{C^{1,\alpha}}}\Abs{v-w}_{C^{1,\alpha}}.
  \end{equation*}
\end{prop}

\begin{remark}
  \label{Rmk_D_SelfAdjoint}
  If $\psi$ is closed,
  then $D$ is self-adjoint;
  indeed, it corresponds to the Hessian of $\Upsilon$.
\end{remark}

\begin{proof}[Proof of \autoref{Prop_Associative_Linearisation+Remainder}]
  To ease notation,
  set $f \coloneq \bP^*(\delta \Upsilon)$.
  Since
  \begin{equation*}
    \Inner{f(v),w}
    = \int_{\Gamma_v} i_w\jmath^*\psi,
  \end{equation*}
  $T_u f \co T_u\Gamma(U) = \Gamma(NP) \to \Gamma(NP)^*$ satisfies
  \begin{equation*}
    \Inner{T_uf(v),w}
    = \int_{\Gamma_u} \sL_vi_w\jmath^*\psi.
  \end{equation*}
  Since
  \begin{equation*}
    f(v)
    = T_0f(v) + \underbrace{\int_0^1 (T_{tv}f-T_0f)(v) \, \rd t}_{\eqcolon \Inner{\sN(v),-}_{L^2}
    },
  \end{equation*}
  it remains to identify $T_0f$ as $D$ and estimate $\sN(v)$.
  
  Choose a frame $(e_1,e_2,e_3)$ on $U$ which restricts to a positive orthonormal frame on $\Gamma_{tu}$ for every $t \in [0,1]$.
  Denote by $\nabla$ the Levi-Civita connection on of $\jmath^*g$ on $U$.
  To ease notation, henceforth suppress $\jmath$.
  Since $\nabla$ is torsion-free,
  \begin{equation}
    \label{Eq_IntegrandInHessUpsilon}
    \begin{split}
      \paren*{\sL_vi_w\psi}(e_1,e_2,e_3)
      &=
      \psi(\nabla_wv,e_1,e_2,e_3) + \Inner{\tau v,w} \phi(e_1,e_2,e_3). \\
      &
      + \psi(w,\nabla_{e_1}v,e_2,e_3)
      + \psi(w,e_1,\nabla_{e_2}v,e_3)
      + \psi(w,e_1,e_2,\nabla_{e_3}v).
    \end{split}
  \end{equation}
  A moment's thought derives the asserted estimate on $\sN$ from this;
  cf.~\cite[Remark 3.5.5]{McDuff2012}.
  
  Since $P$ is associative,
  on $P = \Gamma_0$,
  the first term in \autoref{Eq_IntegrandInHessUpsilon} vanishes and the second equals $\Inner{\tau^\perp v,w}$.
  To digest the second line of \autoref{Eq_IntegrandInHessUpsilon},
  define the cross-product $-\times- \co TY\otimes TY \to TY$ and the associator $[-,-,-] \co TY\otimes TY\otimes TY \to TY$ by
  \begin{equation*}
    \Inner{u\times v,w} \coloneq \phi(u,v,w)
    \qandq
    \psi(u,v,w,x) \coloneq \Inner{[u,v,w],x}.
  \end{equation*}
  These are related by
  \begin{equation*}    
    [u,v,w] = (u\times v)\times w + \Inner{v,w}u - \Inner{u,w}v.
  \end{equation*}
  Therefore,
  \begin{equation*}
    \psi(w,\nabla_{e_i}v,e_j,e_k)
    =
    -\Inner{w,(e_j\times e_k)\times \nabla_{e_i}v};
  \end{equation*}
  cf.~\cite[§4]{Salamon2010}.
  Since $P$ is associative,
  $e_i \times e_j = \sum_{k=1}^3 \epsilon_{ij}^{\phantom{ij}k}e_k$.
  Therefore,
  the second line of \autoref{Eq_IntegrandInHessUpsilon} is
  \begin{equation*}
    -\sum_{a=1}^3 \Inner{e_a \times \nabla_{e_a}v,w} = -\Inner{\gamma\nabla v,w}.
    \qedhere
  \end{equation*}
\end{proof}

In \autoref{Ex_Model_Associative},
the operator $D$, governing the infinitesimal deformation theory of $P = P_{[\eta]}$, can be understood rather concretely.

\begin{example}
  \label{Ex_Model_Associative_D}
  Assume the situation of \autoref{Ex_Model_Associative} with $\mu = \id_\Sigma$.
  Evidently,
  \begin{equation*}
    TP_{[\eta]} = \R\xi \oplus T\Sigma
    \qandq
    NP_{[\eta]} = (\R \xi)^\perp \oplus N\Sigma.
  \end{equation*}
  Direct inspection reveals that $\gamma(- \cdot \xi^\flat)$ defines a complex structure $i$ on $(\R\xi)^\perp$ and agrees with $-I_\xi$ on $N\Sigma$;
  moreover, for $\zeta\cdot v^\flat \in \Hom(T\Sigma,(\R\xi)^\perp)$
  \begin{equation*}
    \gamma(\zeta \cdot v^\flat) = I_\zeta v \in N\Sigma
  \end{equation*}
  A moment's thought shows that
  \begin{equation*}
    \gamma(\zeta \cdot v^\flat I_\xi)
    =
    I_\xi\gamma(\zeta \cdot v^\flat)
    =
    \gamma(i\zeta\cdot v^\flat).
  \end{equation*}
  Therefore,
  the restriction of $\gamma$ to $\Hom(T\Sigma,(\R\xi)^\perp)$ is the composition of
  a complex linear isomorphism
  \begin{equation*}
    \kappa \co \overline\Hom_\C(T\Sigma,(\R\xi)^\perp) \iso N\Sigma
  \end{equation*}
  and the projection $(-)^{0,1} \co \Hom(T\Sigma,(\R\xi)^\perp) \to \overline\Hom_\C\paren{T\Sigma,(\R\xi)^\perp}$ defined by $A^{0,1} \coloneq \frac12(A + IAI)$.
  Therefore,
  \begin{equation*}
    D = D_{P_{[\eta]}}
    =
    (-i\oplus I_\xi)\cdot \del_\xi
    -
    \begin{pmatrix}
      0 & \delbar^*\kappa^* \\
      \kappa\delbar & 0
    \end{pmatrix}
  \end{equation*}
  with the Cauchy--Riemann operator $\delbar \co C^\infty(\Sigma,(\R\xi)^\perp) \to \Gamma\paren[\big]{\overline\Hom_\C\paren{T\Sigma,(\R\xi)^\perp}}$ defined by
  \begin{equation*}
    (\delbar f)(v) \coloneq (
    \rd f)^{0,1}(v) = \frac12\paren{\nabla_v f + i\nabla_{I_\xi v} f}
  \end{equation*}
  and $\delbar^*$ denoting its formal adjoint.
  In particular,
  \begin{equation*}
    \ker D_{P_{[\eta]}} \iso (\R\xi)^\perp \oplus \rH^{0,1}\paren{\Sigma,(\R\xi)^\perp}.
    \qedhere
  \end{equation*}
\end{example}

If $\coker D = 0$, then $P$ is \defined{unobstructed} and stable under perturbations of the $\Gtwo$--structure $\phi$.
In \autoref{Ex_Model_Associative}, $P_{[\eta]}$ is \emph{never} unobstructed,
but does satisfy the following if $\Sigma = S^2$ because $(\R\xi)^\perp = T_{[\eta]}M$.%
\footnote{%
  If $P_{[\eta]}$ is multiply covering,
  then the underlying embedded associative submanifold might be unobstructed;
  see \autoref{Rmk_MultipleCover_Unobstructed}.  
}

\begin{definition}
  \label{Def_MorseBott}
  A smooth map $\bP \co M \to \sS$ is a \defined{Morse--Bott} family of ($\phi$--)associative submanifolds if it is an immersion and for every $x \in M$
  \begin{equation*}
    (\delta\Upsilon)_{\bP(x)} = 0
    \qandq
    \ker D_{\bP(x)} = \im T_x\bP.
    \qedhere    
  \end{equation*}
\end{definition}

Morse--Bott families of $\phi$--associative submanifolds are not stable under small deformations of the $\Gtwo$--structure;
however, the hypothesis of \autoref{Lem_Trivial} can be arranged.
Most of the remainder of this section is devoted to establishing this.
Henceforth, the choice of $\Gtwo$--structure $\phi \in \Omega^3(Y)$ made at the beginning of this section shall be undone.

\begin{definition}
  \label{Def_FamilyStructures}
  Let $\bP_0 \co M \to \sS$ be a smooth map.
  Consider the fibre bundle
  \begin{equation*}
    p \co \underline\bP_0 \coloneq \coprod_{x \in M} \bP_0(x) \incl M \times Y
    \to
    M.
  \end{equation*}
  \begin{enumerate}
  \item
    The \defined{normal bundle} of $\bP_0$ is the vector bundle
    \begin{equation*}
      q \co N\bP_0 \coloneq \coprod_{x \in M} N\P_0(x) \to \underline\bP_0.
    \end{equation*}
    There is a canonical isomorphism $N\bP_0 \iso N\underline\bP_0 \coloneq T(M\times Y)/T\underline\bP_0$.
  \item
    A \defined{tubular neighborhood} of $\bP_0$ is a tubular neighborhood $\bm{\jmath} \co \bU \imm M\times Y$ of $\underline\bP_0$ with $\pr_M \circ \bm{\jmath} = p \circ q$.
    In particular, for every $x \in M$, $\bm\jmath$ induces a tubular neighborhood $\jmath_x \co U_x \imm Y$ of $\bP_0(x)$.
  \item
    \label{Def_FamilyStructures_T}
    Consider the vector bundle
    \begin{equation*}
      r \co \bE \coloneq \coprod_{\bP_0 \in C^\infty(M,\sS)} \Gamma\paren[\big]{\Hom(TM,\bP_0^*T\sS)} \to C^\infty(M,\sS).
    \end{equation*}
    Differentiation defines a section $T \in \Gamma(\bE)$.
    Let $\bm\jmath \co \bU \imm M \times Y$ be a tubular neighborhood of $\bP_0$.
    The map $\bQ_{\bm\jmath} \co \Gamma(\bU) \to C^\infty(M,\sS)$ defined by
    \begin{equation*}
      \bQ_{\bm\jmath}(v)(x) \coloneq \bQ_{\jmath_x}(v) \qwithq
      v_x \coloneq v|_{\bP_0(x)}
    \end{equation*}
    is (the inverse of) a chart on $C^\infty(M,\sS)$.
    Within this chart $\bE$ is trivialised and $T$ is identified with a smooth map $\bT = \bT_{\bm\jmath} \in C^\infty\paren[\big]{\Gamma(\bU),\Gamma\paren{\Hom(p^*TM,N\bP_0)}}$;
    that is:
    the diagram
    \begin{equation*}
      \begin{tikzcd}
        \Gamma(\bU) \times \Gamma(\Hom(p^*TM,N\bP_0)) \ar{r} \ar[shift left]{d} & \bE \ar[shift right]{d}[swap]{r} \\
        \Gamma(\bU) \ar{r}{\bQ} \ar[shift left]{u}{(\id,\bT)} & C^\infty(M,\sS) \ar[u,swap,"T",shift right]
      \end{tikzcd}
    \end{equation*}
    commutes.
    (See \autoref{Fig_FamilyStructures} and \autoref{Rmk_FamilyStructures_T}.)
  \end{enumerate}
  Henceforth, suppose that $\phi_0$ is a $\Gtwo$--structure and that $\bP_0(x)$ is $\phi_0$--associative for every $x \in M$.
  \begin{enumerate}[resume]
  \item
    Define $D = D_{\bP_0} \co \Gamma(N\bP_0) \to \Gamma(N\bP_0)$ by
    \begin{equation*}
      (D v)|_{\bP_0(x)} \coloneq D_{\bP_0(x)} \paren{v|_{\bP_0(x)}}.
    \end{equation*}
    Set
    \begin{equation*}
      \sV
      \coloneq
      \set{
        v \in \Gamma(N\bP_0)
        :
        v|_{\bP_0(x)} \perp \im T_x\bP_0
        ~\textnormal{for every}~ x \in M
      }      
    \end{equation*}
    with $\perp$ denoting $L^2$ orthogonality.
    Denote by $D^\perp = D_{\bP_0}^\perp \co \sV \to \sV$ the map induced by $D$.
  \item
    Let $\bm\jmath \co \bU \imm M \times Y$ be a tubular neighborhood of $\bP_0$.
    Define $\sN = \sN_{\bm\jmath} \co C^\infty(\Gamma(\bU),\Gamma(N\bP_0))$ by
    \begin{equation*}
      (\sN v)|_{\bP_0(x)} \coloneq \sN_{\jmath_x} \paren{v}|_{\bP_0(x)}.
      \qedhere
    \end{equation*}
  \end{enumerate}
\end{definition}

\begin{remark}
  \label{Rmk_FamilyStructures_T}
  The upcoming \autoref{Prop_PerturbationMorseBottAssociatives} constructs a perturbation $\bP \coloneq \bQ_{\bm\jmath}(v)$ of $\bP_0$.
  To establish one of the desired properties of $\bP$, it is necessary to compare the derivatives $T\bP$ and $T\bP_0$.
  The purpose of the map $\bT$ is to enable this.
\end{remark}

\begin{figure}
  \centering
  \begin{tikzpicture}
    \draw[thick] (0,0) -- (4,3) node[right] {\footnotesize $\bP_0$};
    \foreach \x in {0,.2,...,4} { 
      \draw[densely dotted,gray] (\x,.75*\x-.5) -- (\x,.75*\x+.5);
    }
    \draw[cyan,thick] plot [smooth] coordinates {(0,-.25) (2,1.75) (4,2.625)} node [right] {$\bQ(v)$};
    \draw[gray] (-.5,-.5) -- (-.5,3.5) node[above] {\footnotesize $Y$};
    \draw[gray] (0,-1) -- (4,-1) node[right] {\footnotesize $M$};
    \node[magenta] at (2,-1)[circle,fill,inner sep=1pt]{};
    \draw[-stealth,magenta,semithick] (2,-1) -- (2.4,-1) node[midway,above] {\footnotesize $\hat x$};
    \node[magenta] at (2,1.75)[circle,fill,inner sep=1pt]{};
    \draw[-stealth,magenta,semithick] (2,1.75) -- (2,2.25) node[above] {\footnotesize $\bT(v)(\hat x)$};
  \end{tikzpicture}  
  \caption{A sketch of the situation of \autoref{Def_FamilyStructures}~\autoref{Def_FamilyStructures_T}.}
  \label{Fig_FamilyStructures}
\end{figure}

\begin{example}
  \label{Ex_Model_Associative_N}
  In the situation of \autoref{Ex_Model_Associative} with $\mu = \id_\Sigma$,
  $M = T^2$,
  $\underline\bP = T^2 \times (S^1 \times \Sigma)$ and $N\bP = TT^2 \oplus N\Sigma$.
  $D_{\bP(x)}$ and $\sN_{\jmath_x}$---for a suitable choice of $\bm\jmath$ and with respect to suitable identifications---are independent of $x \in T^2$.
\end{example}

\begin{definition}
  \label{Def_HolderNorm}
  Let $\bP \co M \to \sS$ be a smooth map.
  Suppose that Riemannian metrics on $M$ and $Y$ are given.
  This induces a Euclidean inner product and an orthogonal covariant derivative $\nabla$ on $N\bP$, and
  an Ehresmann connection on $p \co \underline\bP \to M$.
  Denote by $\nabla^{1,0}$ and $\nabla^{0,1}$ the restriction of $\nabla$ to the horizontal and vertical directions respectively.
  Denote by
  \begin{equation*}
    \fP \coloneq \coprod_{x \in M} C^\infty([0,1],p^{-1}(x))
  \end{equation*}
  the set of vertical paths in $p \co \underline\bP \to M$.
  Denote by $\fP^+ \subset \fP$ the subset of non-constant paths.
  For $\alpha \in (0,1)$ set
  \begin{equation*}
    [v]_{C^0C^{0,\alpha}}
    \coloneq
    \sup_{\gamma \in \fP^+}
    \frac{\abs{\tra_\gamma(v(\gamma(0)))-v(\gamma(1))}}{\ell(\gamma)^\alpha}
    \qandq
    \Abs{v}_{C^0C^{0,\alpha}}
    \coloneq
    \Abs{v}_{C^0} + [v]_{C^0C^{0,\alpha}}
  \end{equation*}
  with $\ell(\gamma)$ denoting the length of $\gamma$ and $\tra_\gamma$ denoting parallel transport along $\gamma$.
  For $k,\ell \in \N_0,\alpha \in (0,1)$ define the norm $\Abs{-}_{C^kC^{\ell,\alpha}}$ on $\Gamma(N\bP)$ by
  \begin{equation*}
    \Abs{v}_{C^kC^{\ell,\alpha}}
    \coloneq \sum_{m=0}^k \sum_{n=0}^\ell
    \Abs{\paren{\nabla^{1,0}}^m\paren{\nabla^{0,1}}^n v}_{C^0C^{0,\alpha}}.
    \qedhere
  \end{equation*}
\end{definition}

\begin{prop}
  \label{Prop_PerturbationMorseBottAssociatives}
  Let $\alpha \in (0,1)$, $\beta,\gamma,c_1,c_2,c_3,c_4,c_5,R > 0$.
  If $2\beta > \gamma$, then there are constants
  $T = T(\alpha,\beta,\gamma,c_1,c_2,c_3,c_4,c_5,R) > 0$ and
  $c_v = c_v(\alpha,\beta,\gamma,c_1,c_2,c_3) > 0$
  with the following significance.
  Let $\phi_0, \phi \in \Omega^3(Y)$ be two $\Gtwo$--structures on $Y$.  
  Let $\bP_0 \co M \to \sS$ be a Morse--Bott family of $\phi_0$--associative submanifolds.
  Let $\bm{\jmath} \co \bU \imm \bV$ tubular neighborhood of $\bP_0$.
  Let $t \in (0,T)$.
  Suppose that:
  \begin{enumerate}
  \item
    \label{Prop_PerturbationMorseBottAssociatives_Tube}
    $B_R(0) \subset U_x$.
  \item
    \label{Prop_PerturbationMorseBottAssociatives_Error}
    $\Abs{\bm\jmath^*(\phi-\phi_0)}_{C^{1,\alpha}(\bU)} \leq c_1 t^\beta$.
  \item
    \label{Prop_PerturbationMorseBottAssociatives_Linear}
    $D^\perp \co \sV \to \sV$ is bijective and
    \begin{equation*}
      \Abs{v}_{C^1C^{1,\alpha}} \leq c_2 t^{-\gamma} \Abs{D_x^\perp v}_{C^1C^{0,\alpha}}.
    \end{equation*}
  \item
    \label{Prop_PerturbationMorseBottAssociatives_NonLinear}
    $\sN \in C^\infty(\Gamma(\bU),\Gamma(N\bP_0))$ satisfies
    \begin{equation*}
      \Abs{\sN(v)-\sN(w)}_{C^1C^{0,\alpha}}
      \leq
      c_3 \paren*{\Abs{v}_{C^1C^{1,\alpha}}+\Abs{w}_{C^1C^{1,\alpha}}}\Abs{v-w}_{C^1C^{1,\alpha}}.
    \end{equation*}
  \item
    \label{Prop_PerturbationMorseBottAssociatives_T}
    For every $\hat x \in T M$ and $v \in \Gamma(\bU)$
    \begin{equation*}
      \abs{\hat x} \leq c_4 \Abs{\bT(0)(\hat x)}_{C^0}
      \qandq
      \Abs{\bT(v) - \bT(0)}_{C^0} \leq c_5\Abs{v}_{C^1C^{1,\alpha}}.
    \end{equation*}
  \end{enumerate}
  In this situation,
  there is a $v \in \Gamma(\bU) \subset \Gamma(N\bP_0)$ with
  $\Abs{v}_{C^1C^{1,\alpha}} \leq c_vt^{\beta-\gamma}$
  such that the map
  \begin{equation}
    \label{Eq_P}
    \bP \coloneq \bQ_{\bm\jmath}(v) \co M \to \sS
  \end{equation}
  satisfies the hypothesis of \autoref{Lem_Trivial} with respect to $\phi$.
  Moreover, if $H$ is a finite group acting on $M$ and $Y$, $\phi_0$ and $\phi$ are $H$--invariant, and $\bm\jmath$ and $\bP_0$ are $H$--equivariant,
  then $\bP$ is $H$--equivariant.
\end{prop}

\begin{proof}  
  To ease notation,
  define $f_0,f \in C^\infty(\Gamma(\bU),\Gamma(N\bP_0))$
  by
  \begin{equation*}
    \Inner{f_0(v)|_{\bP_0(x)},w}_{L^2} \coloneq \Inner{\paren{\bQ_{\jmath_x}^*(\delta\Upsilon^{\psi_0})}(v),w}
    \qandq
    \Inner{f(v)|_{\bP_0(x)},w}_{L^2} \coloneq \Inner{\paren{\bQ_{\jmath_x}^*(\delta\Upsilon^\psi)}(v),w}.
  \end{equation*}
  Denote by $(-)^\perp$ the projection onto $\sV$.
  For $v \in \sV$
  \begin{equation*}
    (D^\perp)^{-1}f(v)^\perp
    =
    v + \underbrace{(D^\perp)^{-1}\paren*{\sN(v) + f(v) - f_0(v)}^\perp}_{\coloneq E(v)}.
  \end{equation*}
  By \autoref{Prop_PerturbationMorseBottAssociatives_Error}, \autoref{Prop_PerturbationMorseBottAssociatives_Linear}, and \autoref{Prop_PerturbationMorseBottAssociatives_NonLinear},
  there is a constant $c_E = c_E(\alpha,\beta,\gamma,c_1,c_2,c_3) > 0$ such that
  for every $r \in (0,R)$ and $v,w \in \overline B_r(0) \subset C^1C^{1,\alpha}\Gamma(N\bP_0)$
  \begin{align*}
    \Abs{E(0)}_{C^1C^{1,\alpha}}
    &
      \leq c_Et^{\beta-\gamma} \qand
    \\
    \Abs{E(v) - E(w)}_{C^1C^{1,\alpha}}
    &
      \leq c_E(r+t^{\beta})t^{-\gamma} \Abs{v-w}_{C^1C^{1,\alpha}}.
  \end{align*}
  Therefore,
  $-E$ defines a contraction on $\overline B_r(0) \subset C^1C^{1,\alpha}\Gamma(N\bP_0)$ provided
  \begin{equation*}
    c_E(r + t^\beta)t^{-\gamma} < 1
    \qandq
    c_E t^{\beta-\gamma} + c_E(r + t^\beta)t^{-\gamma} r \leq r.
  \end{equation*}
  These can be seen to hold for $r \coloneq 2c_Et^{\beta-\gamma}$ and $t \leq T \ll 1$.
  Denote by $v \in \overline B_r(0) \subset C^1C^{1,\alpha}\Gamma(N\bP_0)$ the unique solution of
  \begin{equation*}
    f(v)^\perp = 0.
  \end{equation*}
  By elliptic regularity, $v \in \Gamma(\bU)$.
  
  It remains to prove that $\bP$ defined by \autoref{Eq_P} satisfies the hypothesis of \autoref{Lem_Trivial};
  that is: for every $x \in M$  
  \begin{equation*}
    \ker (\delta\Upsilon^\psi)_{\bP(x)} + \im T_x\bP = T_{\bP(x)}\sS,
  \end{equation*}
  or, equivalently,
  \begin{equation*}
    f_x(v) = 0 \quad\textnormal{or}\quad
    f_x(v) \notin \paren*{\im \bT_x(v)}^\perp.
  \end{equation*}
  Here the subscript $x$ indicates restriction to $\bP_0(x)$.
  By construction, $f_x(v) \in \im \bT_x(0)$.
  Therefore, the hypothesis is satisfied by \autoref{Prop_PerturbationMorseBottAssociatives_T} provided $t \leq T \ll 1$.

  Evidently, this construction preserves $H$--equivariance.
\end{proof}

In \autoref{Ex_Model_Associative} with $\Sigma = S^2$,
the following gives the required estimate on $D$.

\begin{situation}
  \label{Sit_DelT+A}
  Let $X$ be a compact oriented Riemannian manifold.
  Let $V$ be a Euclidean vector bundle over $X$.
  Let $A \co \Gamma(V) \to \Gamma(V)$ be a formally self-adjoint linear elliptic differential operator of first order.
  Denote by $\pi \co \Gamma(V) \to \ker A$ the $L^2$ orthogonal projection onto $\ker A$.
  Let $L > 0$.
  Define $\Pi \co \Gamma\paren{(\R/L\Z) \times X,V} \to \ker A$ by
  \begin{equation*}
    \Pi s \coloneq \fint_0^L \pi(i_t^*s) \, \rd t
  \end{equation*}
  with $i_t(x) \coloneq (t,x)$.
\end{situation}

\begin{remark}
  \label{Rmk_Model_Associative_D=DelT+A}
  In situation of \autoref{Ex_Model_Associative} with $\mu = \id_\Sigma$,
  according to \autoref{Ex_Model_Associative_D}
  \begin{equation*}    
    D_{P_{[\eta]}}
    =
    (-i\oplus I_\xi)\cdot (\del_\xi+A)
    \qwithq
    A
    \coloneq
    \begin{pmatrix}
      0 & i\delbar^*\kappa^* \\
      -\kappa\delbar i & 0
    \end{pmatrix}.
    \qedhere
  \end{equation*}
\end{remark}

\begin{prop}
  \label{Prop_DelT+A}
  In \autoref{Sit_DelT+A},
  for every $\alpha \in (0,1)$
  there is a constant $c = c(A,\alpha) > 0$ such that for every $s \in \Gamma((\R/L\Z) \times X,V)$
  \begin{equation*}
    \Abs{s}_{C^{1,\alpha}}
    \leq
    c\paren[\big]{(L+1)\Abs{(\del_t+A)s}_{C^{0,\alpha}} + \Abs{\Pi s}_{L^\infty}}.
  \end{equation*}
\end{prop}

\begin{proof}
  By interior Schauder estimates (see, e.g., \cite[§3]{Kichenassamy2006})
  \begin{equation*}
    \Abs{s}_{C^{1,\alpha}}
    \leq
    c_1\paren*{\Abs{(\del_t+A)s}_{C^{0,\alpha}} + \Abs{s}_{L^\infty}}.
  \end{equation*}  
  Define $\hat\pi \co \Gamma((\R/LZ) \times X,V) \to \Gamma((\R/LZ) \times X,V)$ by
  \begin{equation*}
    (\hat\pi s)(t,x) \coloneq \paren{\pi(i_t^*s)}(x).
  \end{equation*}
  A contradiction argument proves that  
  \begin{equation*}
    \Abs{(\one-\hat\pi)s}_{L^\infty}
    \leq c_2\Abs{(\del_t+A)(\one-\hat\pi)s}_{C^{0,\alpha}}
    \leq c_2\paren[\big]{\Abs{(\del_t+A)s}_{C^{0,\alpha}} + \Abs{(\del_t+A)\hat\pi s}_{C^{0,\alpha}}};
  \end{equation*}
  cf.~\cite[Proof of Proposition 8.5]{Walpuski2011}.
  As a consequence of the fundamental theorem of calculus
  \begin{equation*}
    \Abs{\hat\pi s}_{L^\infty}
    \leq
    L\Abs{\del_t \hat\pi s}_{L^\infty} + \Abs{\Pi s}_{L^\infty}
    =
    L\Abs{(\del_t+A)\hat\pi s}_{L^\infty} + \Abs{\Pi s}_{L^\infty}.
  \end{equation*}
  Therefore,
  \begin{equation*}
    \Abs{s}_{L^\infty}
    \leq
    c_2\Abs{(\del_t+A)s}_{C^{0,\alpha}}
    + (c_2+L)\Abs{(\del_t+A)\hat\pi s}_{L^\infty}.
  \end{equation*}
  Evidently,
  \begin{equation*}
    \hat\pi (\del_t + A) = (\del_t + A)\hat\pi.
  \end{equation*}
  Therefore,
  \begin{equation*}
    \Abs{(\del_t+A)\hat\pi s}_{C^{0,\alpha}} \leq c_3\Abs{(\del_t+A)s}_{C^{0,\alpha}}.
  \end{equation*}
  The above observations combine to the asserted estimate with $c = c_1(c_2+1)(c_3+1)$. 
\end{proof}


\section{Examples}
\label{Sec_Examples}

The purpose of this section is to construct the associative submanifolds whose existence was promsied in \autoref{Sec_Introduction}.
Here is a construction technique based on \autoref{Prop_PerturbationMorseBottAssociatives} and \autoref{Rmk_Trivial}~\autoref{Rmk_Trivial_H}.

\begin{prop}
  \label{Prop_ExistenceOfAssociatives_Orbifold}
  Let $\fR = (\Gamma_\alpha,G_\alpha,\rho_\alpha;R_\alpha,\jmath_\alpha;\hat X_\alpha,\hat\bomega_\alpha,\hat\rho_\alpha,\tau_\alpha)_{\alpha\in A}$ be resolution data for a closed flat $\Gtwo$--orbifold $(Y_0,\phi_0)$.
  Denote by $(Y_t,\phi_t)_{t \in (0,T_0)}$ the family of closed $\Gtwo$--manifolds obtained from the generalised Kummer construction discussed in \autoref{Sec_GeneralisedKummer}.
  Let $\star \in A$, $\hat \xi \in S^2 \subset \Im \H$, $L > 0$, and $\Sigma \subset X_\star$.
  Set $\xi \coloneq L\hat\xi$, $\Lambda_\star \coloneq G_\star \cap \Im \H < \Im \H$, $M_\star \coloneq \paren{\Im\H/\R\xi}/\paren{\Lambda_\star/\Z\xi}$, and $H_\star \coloneq G_\star/\Lambda_\star$.
  Denote by $n_f$ the number of singularities of the orbifold $M_\star/H_\star$ (see~\autoref{Rmk_Model_Associative_Possibilities}).
  Denote by $\bI_\star$ the hypercomplex structure on $X_\star$.
  Suppose that:
  \begin{enumerate}
  \item
    \label{Prop_ExistenceOfAssociatives_Orbifold_Sigma}
    $\Sigma$ is a closed $I_{\star,\hat\xi}$--holomorphic curve.
    $\Sigma \iso S^2$.
  \item
    \label{Prop_ExistenceOfAssociatives_Orbifold_Xi}
    $\xi \in \Lambda_\star$ is primitive.
    $\xi \in Z(G_\star)$.
  \item
    \label{Prop_ExistenceOfAssociatives_Orbifold_Rho}
    $\rho_\star(g)(\Sigma) = \Sigma$ for every $g \in G_\star$, and
    $\rho(\xi)|_\Sigma = \id_\Sigma$.
  \end{enumerate}
  In this situation,
  there is a constant $T_1 \in (0,T_0]$ and for every $t \in (0,T_1)$ there are at least $n_f$ distinct associative submanifolds in $(Y_t,\phi_t)$ representing the homology class $\beta \coloneq \eta_\star([P_{[0]}]) \in \rH_3(Y_t,\Z)$ with $\eta_\star$ as in \autoref{Rmk_HomologyMaps} and $P_{[0]} \subset \hat Y_{\star,t}$ as in  \autoref{Ex_Model_Associative}.
\end{prop}

\begin{proof}
  For every $[\eta] \in M_\star/H_\star$ and $t \ll 1$,
  \autoref{Ex_Model_Associative} constructs a $t^{-3}\tilde\phi_t$--associative submanifold $P_{[\eta]} \imm \hat Y_{\star,t}^\circ \setminus \hat V_{\star,t}$.
  This defines an $H_\star$--invariant and $K$--equivariant Morse--Bott family $\bP_0 \co M_\star \to \sS$ of $t^{-3}\tilde \phi_t$--associative submanifolds; see \autoref{Ex_Model_Associative_D}.
  With respect to $t^{-2}\tilde g_t$ these submanifolds are isometric to $(\R/t^{-1}L\Z) \times \Sigma$.
  
  The hypotheses of \autoref{Prop_PerturbationMorseBottAssociatives} are satisfied for the choices $\phi_0 = t^{-3}\tilde\phi_t$, $\phi = t^{-3}\phi_t$, $\alpha \in (0,1/16)$, $\beta = 5/2$, $\gamma = 1$, and $c_1,c_2,c_3,c_4,c_5,R >0$ of secondary importance:  
  \autoref{Prop_PerturbationMorseBottAssociatives_Tube} holds for $0 < R \ll 1$.  
  \autoref{Prop_PerturbationMorseBottAssociatives_Error} holds by \autoref{Thm_PerturbG2Structure}.
  Because of \autoref{Thm_PerturbG2Structure} it suffices to verify \autoref{Prop_PerturbationMorseBottAssociatives_Linear}, \autoref{Prop_PerturbationMorseBottAssociatives_NonLinear}, and \autoref{Prop_PerturbationMorseBottAssociatives_T} with respect to $t^{-2}\tilde g_t$.
  \autoref{Rmk_Model_Associative_D=DelT+A} and \autoref{Prop_DelT+A} imply \autoref{Prop_PerturbationMorseBottAssociatives_Linear} with respect to $t^{-2}\tilde g_t$ because $\sV$ is defined to be the kernel of $\Pi$.
  Because of \autoref{Rmk_T-3Phi} (the proof of) \autoref{Prop_Associative_Linearisation+Remainder} implies \autoref{Prop_PerturbationMorseBottAssociatives_NonLinear}.
  For a suitable choice of $\bm\jmath$,
  \autoref{Prop_PerturbationMorseBottAssociatives_T} holds with respect to $t^{-2}\tilde g_t$ by direct inspection; cf.~\autoref{Ex_Model_Associative_N}.

  For $t \in (0,T_{1/2})$ the resulting $H_\star$--invariant map $\P \co M_\star \to \sS$ satisfies the hypothesis of \autoref{Lem_Trivial}.
  By \autoref{Rmk_Trivial}~\autoref{Rmk_Trivial_H},
  every isolated fixed-point of the action of $H_\star$ on $M_\star$ is a zero of $\bP^*(\delta\Upsilon^{\phi_t})$.    
  If $t < T_1 \ll 1$, then these map to $n_f$ pairwise distinct elements of $\sS$.
  By \autoref{Lem_Trivial}, each one of these is a $\phi_t$--associative submanifold.
\end{proof}

\begin{remark}
  \label{Rmk_MultipleCover_Unobstructed}
  If $x \in M_\star$ corresponds to an orbifold point $[x] \in M_\star/H_\star$,
  then $P_0 \coloneq \bP_0(x)$ and $\bP(x)$ are multiply covering and their deck transformation group contain the isotropy group $\Gamma$ of $[x]$.
  The embedded associative submanifold $\check P_0 \coloneq P_0/\Gamma$ is unobstructed;
  indeed:
  \begin{equation*}
    \ker D_{\check P_0} = (\ker D_{P_0})^\Gamma = (T_xM_\star)^\Gamma = 0.
  \end{equation*}
  Similarly, if $x \in M_\star$ is an isolated fixed-point of the action of $K$,
  then $P_0 \coloneq \bP_0(x)$ is preserved by the action of $K$ and
  \begin{equation*}
    \paren{\ker D_{P_0}}^K = (T_xM_\star)^K = 0.
  \end{equation*}
  This can be used to give a somewhat simpler proof of most of \autoref{Prop_ExistenceOfAssociatives_Orbifold} avoiding the use of \autoref{Prop_PerturbationMorseBottAssociatives}.
\end{remark}
  
\begin{example}
  \label{Ex_ExistenceOfAssociatives_C2}
  \citet[Examples 4, 5, 6]{Joyce1996a} constructs 7 examples of closed flat $\Gtwo$--orbifolds $(Y_0,\phi_0)$ whose singular set has components $S_\alpha$ ($\alpha \in A$).
  $A$ is a disjoint union $A = A^0 \amalg A^1$ with $A^1 \neq \emptyset$.
  For $\alpha \in A^0$, $S_\alpha$ is isometric to $T^3 = \R^3/\Z^3$.
  For $\alpha \in A^1$, $S_\alpha$ is isometric to $T^3/C_2$.
  Here is a more precise description.
  For $\alpha \in A$ set $\Gamma_\alpha \coloneq C_2$.
  For $\alpha \in A^0$ set $G_\alpha \coloneq \Lambda = \Span{i,j,k} < \Im\H$ and denote by $\rho_\alpha \co G_\alpha \to \Isom(\H/\Gamma_\alpha)$ the trivial homomorphism.
  For $\alpha \in A^1$   
  let $G_\alpha < \SO(\Im\H)\ltimes\Im \H$ be generated by $\Lambda$ and $(R_2,\frac{i}2)$, and
  define $\rho_\alpha \co G_\alpha \to G_\alpha/\Lambda \to \SO(\H)^{\Gamma_\alpha} \incl \Isom(\H/\Gamma_\alpha)$ by
  \begin{equation*}
    \rho_\alpha\paren[\big]{R_2,\tfrac{i}2}[q] \coloneq [-iqi].
  \end{equation*}
  For every $\alpha \in A$ there is an open embedding
  $\jmath_\alpha \co \paren[\big]{\Im \H \times \paren{B_{R_\alpha}(0)/\Gamma_\alpha}}/G_\alpha \to Y_0$
  as in \autoref{Def_ResolutionData}~\autoref{Def_ResolutionData_Model}.

  These can be extended to resolution data $\fR$ for $(Y_0,\phi_0)$ with the aid of the Gibbons--Hawking construction discussed in \autoref{Rmk_AkALE}.  
  According to \autoref{Rmk_AkALE}~\autoref{Rmk_AkALE_Decay}, $(X_{\bm{0}},\bomega_{\bm{0}})$ is $\H/C_2$ with the standard hyperkähler form.
  If $\bzeta = [\zeta,-\zeta] \in \Delta^\circ$,
  then $(X_\bzeta,\bomega_\zeta)$ is a hyperkähler manifold and \autoref{Rmk_AkALE}~\autoref{Rmk_AkALE_Decay} provides $\tau \co X_{\bzeta}\setminus K_\bzeta \to X_{\bm{0}}$.
  Therefore, completing the resolution data for $\alpha \in A^0$ amounts to a choice of $\zeta_\alpha \in \Delta^\circ$
  
  For $\alpha \in A^1$ the situation is slightly complicated by the fact that $\rho_\alpha$ is non-trivial.
  The involution $R(q) \coloneq -iqi$ lies in $\SO(\H)^{\Gamma_\alpha}$ and $\Lambda^+R = R_2$.
  By \autoref{Rmk_AkALE}~\autoref{Rmk_AkALE_Symmetry},
  requiring that $R$ lifts to $X_\bzeta$ imposes the constraint that $\bzeta_\alpha \in (\Delta^\circ)^{R_2}$.
  Therefore, completing the resolution data for $\alpha \in A^1$ amounts to a choice of $\bzeta_\alpha \in  (\Delta^\circ)^{R_2}$.
  A moment's thought shows that
  \begin{equation*}
    (\Delta^\circ)^{R_2}
    =
    \set[\big]{ [\zeta,-\zeta] \in \Delta^\circ : \zeta \in \R i \cup (\R i)^\perp }.
  \end{equation*}
  If $\bzeta = [\zeta,-\zeta]$ with $\zeta \in \R i$,
  then the segment joining $\zeta$ and $-\zeta$ lifts to an $I_i$--holomorphic curve $\Sigma \iso S^2$.
  Therefore, for the corresponding choices of $\fR$,
  \autoref{Prop_ExistenceOfAssociatives_Orbifold} with $\hat\xi = i$ and $L = 1$ exhibits $4$ associative submanifolds in $(Y_t,\phi_t)$ for every $t \in (0,T_1)$.
\end{example}

\begin{example}
  \label{Ex_ExistenceOfAssociatives_C3}
  \citet[Examples 15, 16]{Joyce1996a} constructs two examples of closed flat $\Gtwo$--orbifolds $(Y_0,\phi_0)$ whose singular set has components $S_\alpha$ ($\alpha \in A$).
  $A$ is a disjoint union $A = A^0 \amalg A^1$ with $A^1 = \set{\star}$.
  The situation is analogous to that in \autoref{Ex_ExistenceOfAssociatives_C2} except that $\Gamma_\star \coloneq C_3$.

  Completing the resolution data for $\star$ amounts to a choice of $\bzeta_\star \in (\Delta^\circ)^{R_2}$.
  A moment's thought shows that
  \begin{equation*}
    (\Delta^\circ)^{R_2}
    =
    \set[\big]{
      [\zeta_1,\zeta_2,\zeta_3] \in \Delta^\circ
      :
      \zeta_1 \in \R i ~\textnormal{and}~ R_2\zeta_2 = \zeta_3
    }
    \cup
    \set[\big]{
      [\zeta_1,\zeta_2,\zeta_3] \in \Delta^\circ
      :
      \zeta_1,\zeta_2,\zeta_3 \in \R i
    }.
  \end{equation*}
  If $\bzeta = [\zeta_1,\zeta_2,\zeta_3]$ is in the latter component and $\zeta_2$ is contained the the segment joining $\zeta_1$ and $\zeta_2$,
  then the segment joining $\zeta_1$ and $\zeta_2$ and the segment joining $\zeta_2$ and $\zeta_3$ lift to $I_{\star,i}$--holomorphic curves $\Sigma_1,\Sigma_2 \iso S^2$ and \autoref{Prop_ExistenceOfAssociatives_Orbifold}~\autoref{Prop_ExistenceOfAssociatives_Orbifold_Xi} holds.
  Therefore,
  for the corresponding choices of $\fR$,
  \autoref{Prop_ExistenceOfAssociatives_Orbifold} with $\hat\xi = i$ and $L = 1$ exhibits $8 = 2\cdot 4$ associative submanifolds in $(Y_t,\phi_t)$ for every $t \in (0,T_1)$.  
\end{example}

\begin{example}
  \label{Ex_ExistenceOfAssociatives_C2'}
  \citet[§5.3.4]{Reidegeld2017} constructs an example of a closed flat $\Gtwo$--orbifold $(Y_0,\phi_0)$ whose singular set has $16$ components $S_\alpha$ ($\alpha \in A$).
  For every $\alpha \in A$, $S_\alpha$ is isometric to $T^3/C_2^2$.
  Here is a more precise description.
  For $\alpha \in A$
  set $\Gamma_\alpha \coloneq C_2$,
  let $G_\alpha < \SO(\Im\H)\ltimes\Im \H$ be generated by $\Lambda \coloneq \Span{i,j,k}$, $(R_+,\frac{i+k}2)$, and $(R_-,\frac{j}{2})$,
  define $\rho_\alpha \co G_\alpha \to G_\alpha/\Lambda \to \SO(\H)^{\Gamma_\alpha} \incl \Isom(\H/\Gamma_\alpha)$ by
  \begin{equation*}
    \rho_\alpha\paren[\big]{R_+,\tfrac{i+k}2}[q] \coloneq [iqi] \qandq
    \rho_\alpha\paren[\big]{R_-,\tfrac{j}2}[q] \coloneq [jqj].
  \end{equation*}
  These act on $\Im\H$ as $R_+$ and $R_-$.
  For every $\alpha \in A$ there is an open embedding
  $\jmath_\alpha \co \paren[\big]{\Im \H \times \paren{B_{R_\alpha}(0)/\Gamma_\alpha}}/G_\alpha \to Y_0$
  as in \autoref{Def_ResolutionData}~\autoref{Def_ResolutionData_Model}.

  Completing the resolution data for $\alpha \in A$ amounts to a choice of $\bzeta_\alpha \in (\Delta^\circ)^{R_+,R_-}$.
  A moment's thought shows that
  \begin{equation*}
    (\Delta^\circ)^{R_+,R_-} = \set[\big]{ [\zeta,-\zeta] \in \Delta^\circ : \zeta \in \R i }.
  \end{equation*}
  Therefore,
  for \emph{every choice} of $\fR$,
  \autoref{Prop_ExistenceOfAssociatives_Orbifold} with $\hat\xi = i$ and $L = 1$ exhibits $32 = 16 \cdot 2$ associative submanifolds in $(Y_t,\phi_t)$ for every $t \in (0,T_1)$.
\end{example}

\begin{example}
  \label{Ex_ExistenceOfAssociatives_Dic2}
  Here is an example that involves non-cyclic $\Gamma$ and requires the use of \autoref{Rmk_Kronheimer_ALE}.
  \citet[§5.3.4]{Reidegeld2017} constructs an example of a closed flat $\Gtwo$--orbifold $(Y_0,\phi_0)$ whose singular set has $7$ components $S_\alpha$ ($\alpha \in A$).
  The situation is analogous to that in \autoref{Ex_ExistenceOfAssociatives_C2'} except that $A = A^1 \amalg A^2 \amalg A^3$ and
  \begin{equation*}
    \Gamma_\alpha \coloneq
    \begin{cases}
      C_2 & \textnormal{if}~\alpha \in A^1 \\
      C_4 & \textnormal{if}~\alpha \in A^2 \\
      \Dic_2 & \textnormal{if}~\alpha \in A^3.
    \end{cases}
  \end{equation*}  

  Completing the resolution data for $\alpha \in A^1$ is identical to \autoref{Ex_ExistenceOfAssociatives_C2'}.
  For $\alpha \in A^2$ a moment's thought shows that
  \begin{align*}
    (\Delta^\circ)^{R_+,R_-}
    &
      = \set[\big]{ [\zeta,R_+\zeta,R_-\zeta,R_+R_-\zeta] \in \Delta^\circ : \zeta \notin \R i \cup \R j \cup \R k} \\
    &
      \cup
      \set[\big]{ [\zeta_1,\zeta_2,-\zeta_1,-\zeta_2] \in \Delta^\circ : \zeta_1,\zeta_2 \in \R i \cup \R j \cup \R k}.
  \end{align*}
  If $\bzeta$ is in the latter component,
  then $X_\bzeta$ contains $I_{\bzeta,\hat\xi_a}$--holomorphic curves $\Sigma_a \iso S^2$ with $\hat\xi_a = \zeta_a/\abs{\zeta_a}$ ($a=1,2$).
  (In fact, there are more.)
  Therefore,
  for the corresponding choices of $\fR$,
  \autoref{Prop_ExistenceOfAssociatives_Orbifold} with $\hat\xi_a = i$ and $L = 1$ exhibits $4 = 2\cdot 2$ associative submanifolds in $(Y_t,\phi_t)$ for every $t \in (0,T_1)$.  

  To understand the situation for $\alpha \in A^3$,
  recall that the $D_4$ root system is
  \begin{equation*}
    \Phi = \set[\big]{ \pm e_a \pm e_b \in \R^4 : a \neq b \in \set{1,2,3,4} }.
  \end{equation*}
  The standard choice of simple roots is
  \begin{equation*}
    \alpha_1 \coloneq e_1-e_2, \quad
    \alpha_2 \coloneq e_2-e_3, \quad
    \alpha_3 \coloneq e_3-e_4, \qandq
    \alpha_4 \coloneq e_3+e_4.
  \end{equation*}
  The Weyl group $W = S^4 \ltimes C_2^3$ acts by permuting and flipping the signs on an even number of the coordinates of $\R^4$.
  Therefore,
  \begin{equation*}
    \Delta^\circ
    =
    \set[\big]{ [\zeta_1,\zeta_2,\zeta_3,\zeta_4] \in (\Im\H \otimes \R^4)/W : \zeta_a \neq \pm \zeta_b ~\textnormal{for}~ a \neq b \in \set{1,2,3,4} }.
  \end{equation*}
  A little computation reveals that
  \begin{align*}
    (\Delta^\circ)^{R_+,R_-}
    &
      =
      \set[\big]{
      [\zeta,R_+\zeta,R_-\zeta,R_+R_-\zeta] \in \Delta^\circ
      :
      \zeta \notin (\R i)^\perp \cup (\R j)^\perp \cup (\R k)^\perp
      } \\
    &
      \cup
      \set[\big]{
      [\zeta_1,\zeta_2,-\zeta_1,-\zeta_2] \in \Delta^\circ
      :
      \zeta_1,\zeta_2\in {(\R i)^\perp \cup (\R j)^\perp \cup (\R k)^\perp} \setminus \paren{\R i \cup \R j \cup \R k}
      } \\
    &
      \cup
      \set[\big]{
      [\zeta_1,\zeta_2,\zeta_3,\zeta_4] \in \Delta^\circ
      :
      \zeta_1,\zeta_2,\zeta_3,\zeta_4 \in \R i
      ~\textnormal{or}~
      \zeta_1,\zeta_2,\zeta_3,\zeta_4 \in \R j
      ~\textnormal{or}~
      \zeta_1,\zeta_2,\zeta_3,\zeta_4 \in \R k
      }.
  \end{align*}
  If $\bzeta = [\zeta_1,\zeta_2,\zeta_3,\zeta_4]$ is in the latter component,
  then, by \autoref{Rmk_Kronheimer_ALE}~\autoref{Rmk_Kronheimer_ALE_HolomorphicCurves},
  $X_\bzeta$ contains $4$ $I_{\bzeta,\hat\xi}$--holomorphic curves $\Sigma_a \iso S^2$ ($a \in \set{1,2,3,4}$) with $\hat\xi \coloneq \zeta_1/\abs{\zeta_1}$.
  Therefore,
  for the corresponding choices of $\fR$,
  \autoref{Prop_ExistenceOfAssociatives_Orbifold} with $\hat\xi = i$ and $L = 1$ exhibits $8 = 4\cdot 2$ associative submanifolds in $(Y_t,\phi_t)$ for every $t \in (0,T_1)$.
\end{example}

\begin{remark}
  \citet[§5.3.4 and §5.3.5]{Reidegeld2017} constructed further examples of closed flat $\Gtwo$--orbifolds $(Y_0,\phi_0)$ whose singular sets are isometric to $T^3$, $T^3/C_2$, and $T^3/C_2^2$ and whose transverse singularity models are $\H/\Gamma$ with $\Gamma \in \set{C_2,C_3,C_4,C_6,\Dic_2,\Dic_3,2T}$.
  The reader might enjoy analysing these examples with the methods used above.
\end{remark}

Here is a construction technique based on \autoref{Prop_PerturbationMorseBottAssociatives} and \autoref{Rmk_Trivial}~\autoref{Rmk_Trivial_Exact}.

\begin{prop}
  \label{Prop_ExistenceOfAssociatives_K=>Exact}
  Let $\fR = (\Gamma_\alpha,G_\alpha,\rho_\alpha;R_\alpha,\jmath_\alpha;\hat X_\alpha,\hat\bomega_\alpha,\hat\rho_\alpha,\tau_\alpha;\lambda_\alpha,\kappa_\alpha,\hat\kappa_\alpha)_{\alpha\in A}$ be $K$--equivariant resolution data for a closed flat $\Gtwo$--orbifold $(Y_0,\phi_0)$ together with a homomorphism $\lambda \co K \to \Diff(Y_0)$ with respect to which $\phi_0$ is $K$--invariant.
  Denote by $(Y_t,\phi_t)_{t \in (0,T_0)}$ the family of closed $\Gtwo$--manifolds obtained from the $K$--equivaraint generalised Kummer construction discussed in \autoref{Rmk_GeneralisedKummer_KInvariant}.
  Let $\star \in A$, $\hat \xi \in S^2 \subset \Im \H$, $L > 0$, and $\Sigma \subset X_\star$.
  Set $\xi \coloneq L\hat\xi$, $\Lambda_\star \coloneq G_\star \cap \Im \H < \Im \H$, and $M_\star \coloneq \paren{\Im\H/\R\xi}/\paren{\Lambda_\star/\Z\xi}$.
  Denote by $\bI_\star$ the hypercomplex structure on $X_\star$.
  Suppose that \autoref{Prop_ExistenceOfAssociatives_Orbifold_Sigma}, \autoref{Prop_ExistenceOfAssociatives_Orbifold_Xi}, and \autoref{Prop_ExistenceOfAssociatives_Orbifold_Rho} in \autoref{Prop_ExistenceOfAssociatives_Orbifold} hold; and moreover:
  \begin{enumerate}
    \setcounter{enumi}{3}
  \item    
    \label{Prop_ExistenceOfAssociatives_K=>Exact_K}
    $g\star = \star$ for every $g \in K$.
    $\kappa_\star(K) < N_{\SO(\Im\H)\ltimes\Im\H}(\Z\xi)$ and
    $\hat\kappa_\star(g)(\Sigma) = \Sigma$ for every $g \in K$.
  \item
    \label{Prop_ExistenceOfAssociatives_K=>Exact_Vanishing}
    $\Hom(\pi_1(M_\star),\R)^K = 0$.
  \end{enumerate}
  In this situation,
  there is a constant $T_1 \in (0,T_0]$ and for every $t \in (0,T_1)$ there are at least $3$ distinct associative submanifolds in $(Y_t,\phi_t)$ representing the homology class $\beta \coloneq \eta_\star([P_{[0]}]) \in \rH_3(Y_t,\Z)$ with $\eta_\star$ as in \autoref{Rmk_HomologyMaps} and $P_{[0]} \subset \hat Y_{\star,t}$ as in  \autoref{Ex_Model_Associative}.
\end{prop}

\begin{proof}
  The proof is very similar to that of \autoref{Prop_ExistenceOfAssociatives_Orbifold}.
   The additional hypothesis \autoref{Prop_ExistenceOfAssociatives_K=>Exact_K} guarantees that $K$ acts on $M_\star$ and that the map $\bP_0$ is $K$--equivariant.
   Therefore, $\bP$ is $K$--equivariant as well.
   According to \autoref{Rmk_Trivial}~\autoref{Rmk_Trivial_Exact},
   the obstruction to $\bP^*(\delta\Upsilon^{\phi_t})$ being exact is the composite homomorphism
   \begin{equation*}
     \pi_1(M_\star) \xrightarrow{\pi_1(\bP)} \pi_1(\sS) \xrightarrow{\textnormal{sweep}} \rH_4(Y_t) \xrightarrow{\Inner{-,[\psi_t]}} \R.
   \end{equation*}
   The first two homomorphisms are manifestly $K$--equivariant.
   The third homomorphism is $K$--equivariant because $\psi_t$ is $K$--invariant
   (see \autoref{Rmk_GeneralisedKummer_KInvariant}).
   By \autoref{Prop_ExistenceOfAssociatives_K=>Exact_Vanishing},
   the composition vanishes.
   Therefore, $\bP^*(\delta\Upsilon^{\phi_t})$ is exact.
   Since $M_\star \not\iso S^2$, $\bP^*(\delta\Upsilon^{\phi_t})$ has at least $3$ zeros.
\end{proof}
  
\begin{example}
  \label{Ex_ExistenceOfAssociatives_K}
  Set $T^7 \coloneq \R^7/\Z$.
  Define the torsion-free $\Gtwo$--structure $\phi_0$ by
  \begin{gather*}
    \phi_0 \coloneq
    \rd x_1 \wedge \rd x_2 \wedge \rd x_3
    - \sum_{a=1}^3 \rd x_a \wedge \omega_a \quad\textnormal{with} \\
    \omega_1 \coloneq \rd x_4 \wedge \rd x_5 + \rd x_6 \wedge \rd x_7, \quad
    \omega_2 \coloneq \rd x_4 \wedge \rd x_6 + \rd x_7 \wedge \rd x_5, \quad
    \omega_3 \coloneq \rd x_4 \wedge \rd x_7 + \rd x_5 \wedge \rd x_6.
  \end{gather*}  
  Define $\iota_1,\iota_2,\iota_3,\lambda \in \Isom(T^7)$ by
  \begin{align*}
    \iota_1[x_1,\ldots,x_7] &\coloneq \sqparen{x_1,x_2,x_3,-x_4,-x_5,-x_6,-x_7}, \\
    \iota_2[x_1,\ldots,x_7] &\coloneq \sqparen{x_1,-x_2,-x_3,x_4,x_5,\tfrac12-x_6,-x_7}, \\
    \iota_3[x_1,\ldots,x_7] &\coloneq \sqparen{-x_1,x_2,-x_3,x_4,\tfrac12-x_5,x_6,\tfrac12-x_7}, \qand \\
    \lambda[x_1,\ldots,x_7] &\coloneq \sqparen{x_1,-x_2,-x_3,x_4,x_5,-x_6,-x_7}.
  \end{align*}
  $(Y_0 \coloneq T^7/\Span{\iota_1,\iota_2,\iota_3},\phi_0)$ is the closed flat $\Gtwo$--oribfold from \cite[Example 3]{Joyce1996a}.
  Its singular set has $12 = 3\cdot 4$ components $S_\alpha$ ($\alpha \in A = A^1 \amalg A^2 \amalg A^3$).
  Here $A^a$ groups those components arising from the fixed-point set of $\iota_a$.
  The situation is analogous to that in \autoref{Ex_ExistenceOfAssociatives_C2} except that, for every $\alpha \in A$, $S_\alpha$ isometric to $T^3$ and $G_\alpha \coloneq \Lambda = \Span{i,j,k} < \Im \H$.
  
  The involution $\lambda$ decends to $Y_0$:
  it can be identified with an action of $C_2$ on $Y_0$ as in \autoref{Rmk_GeneralisedKummer_KInvariant}.
  The induced action on $A$ fixes the elements of $A^1$ and permutes those of $A^2$ and $A^3$. 
  Completing the $C_2$--equivariant resolution data for $\alpha \in A^2 \amalg A^3$ presents no difficult.
  For $\alpha \in A^1$,
  $\lambda_\alpha = R_2$ and $\kappa_\alpha(q) = -iqi$ as in \autoref{Ex_ExistenceOfAssociatives_C2}.
  Therefore, completing the resolution data for $\alpha \in A^1$ amounts to a choice of
  \begin{equation*}
    \bzeta_\alpha
    \in
    (\Delta^\circ)^{R_2}
    =
    \set[\big]{ [\zeta,-\zeta] \in \Delta^\circ : \zeta \in \R i \cup (\R i)^\perp }.
  \end{equation*}
  If $\bzeta_\alpha = [\zeta_\alpha,-\zeta_\alpha]$ with $\zeta_\alpha \in \R i$,
  then the hypotheses of \autoref{Prop_ExistenceOfAssociatives_K=>Exact} satisfied with $\hat\xi = i$, $L = 1$, and $\Sigma$ as in \autoref{Ex_ExistenceOfAssociatives_C2};
  indeed: $C_2$ acts $T^2$ by $[x_2,x_3] \mapsto [-x_2,x_3]$;
  hence: $\Hom(\pi_1(T^2),\R)^{C_2} = 0$.
  This exhibits upto $12 = 4 \cdot 3$ associative submanifolds in $(Y_t,\phi_t)$ for every $t \in (0,T_1)$ depending on the choice of $C_2$--equivariant resolution data.
\end{example}

\begin{remark}
  \label{Rmk_JoyceKarigiannis_Potential}  
  If $X$ is a $K3$ surface with a non-symplectic involution $\tau$,
  then the fixed-point locus $X^\tau$ (typically) contains a surface of genus $g \neq 1$ \cite[§4]{Nikulin1983}.
  The twisted connected sum construction \cite{Kovalev2003,Kovalev2011,Corti2012a}---in fact: a trivial version of thereof---produces closed $\Gtwo$--orbifolds $(Y_0,\phi_0)$ from a matching pairs $(\Sigma_\pm,\tau_\pm)$ of non-symplectic $K3$ surfaces.
  The singular set of $Y_0$ is $S^1 \times M$ with $M \coloneq X_+^\tau \cup X_-^\tau$ and the transverse singularity model is $\H/C_2$.
  An extension of the generalised Kummer construction due to \citet{Joyce2017} resolves $Y_0$ into a family $(Y_t,\phi_t)_{t\in(0,T_0)}$ of closed $\Gtwo$--manifolds.
  It seems plausible that an extension of the techniques in the present article could produce $\bP \co M \to \sS$ satisfying the hypothesis of \autoref{Lem_Trivial}.
  Since (typically) $\chi(M) \neq 0$,
  this would produce associatives in \citeauthor{Joyce2017}'s $\Gtwo$--manifolds.
\end{remark}

\begin{remark}
  \label{Rmk_Coassociatives}
  It is also possible to construct coassociative submanifolds in $\Gtwo$--manifolds obtained from the generalised Kummer construction using similar techniques.
  In fact, the situation is quite a bit simpler because the deformation theory of coassociative submanifolds is always unobstructed \cite[§4]{McLean1998}.
\end{remark}


\printreferences

\end{document}
